\documentclass[11pt,reqno]{amsart}
\usepackage[a4paper, total={6.3in, 8.5in}]{geometry}

\usepackage{amsfonts,amscd,amsmath,amsthm,amssymb,amsfonts,latexsym,}
\usepackage[usenames,dvipsnames]{xcolor}
\usepackage{enumerate}
\usepackage{manfnt}
\usepackage[backref]{hyperref}
\usepackage[nameinlink]{cleveref}
\hypersetup{colorlinks=true,linkcolor=blue,urlcolor=blue,citecolor=blue}

\title{Generalizations of The Finite Height Criterion for Local Tabularity}

\author{Ilya B.~Shapirovsky}

\address{New Mexico State University, USA}
  \email{ilshapir@nmsu.edu}

\newcommand\hide[1]{{\color{gray} #1}}
\renewcommand\hide[1]\empty

\newcommand\IS[1]{{\bf IS}: {\color{teal} #1}}
\renewcommand\IS[1]\empty

\newcommand\todo[1]{ [~ {\bf {\color{BrickRed} #1 }]}}
\renewcommand\todo[1]\empty
\newcommand\improve[1]{ [~ {\color{BlueGreen}\noindent{\bf Improve:} #1 }]}
\renewcommand\improve[1]\empty
\newcommand\later[1]{ [~ {\color{Red}\noindent{\bf Backlog:} #1 }]}
\renewcommand\later[1]\empty

\theoremstyle{definition}
\newtheorem{theorem}{Theorem}[section]
\newtheorem{proposition}[theorem]{Proposition}
\newtheorem{lemma}[theorem]{Lemma}
\newtheorem{corollary}[theorem]{Corollary}

\newtheorem{definition}[theorem]{Definition}
\newtheorem*{definition*}{Definition}
\newtheorem{remark}[theorem]{Remark}
\newtheorem{example}[theorem]{Example}
\newtheorem{problem}{Problem}

\def\Al{\mathrm{\Omega}}
\def\AlA{\Al}
\def\AlB{\mathrm{\AlA_0}}

\newcommand\framets[1]{#1}
\def\frI{\framets{I}}
\def\frF{\framets{F}}

\def\mM{\framets{M}}

\newcommand\LSuml[2]{{\textstyle {\sum\limits^{\lex}}_{#1}{#2}}}
\def\tiff{\text{ iff }}
\def\clI{\mathcal{I}}
\def\clF{\mathcal{F}}

\def\Log{\myoper{Log}}
\newcommand\myoper[1]{\mathop{\myopts{#1}}}
\newcommand\myopts[1]{\mathrm{#1}}
\def\lex{\mathrm{lex} }

\def\Di{\lozenge}
\def\Dih{{\Di^{\mathrm{h}}}}
\def\Div{{\Di^{\mathrm{v}}}}
\def\Boxh{\Box^{\mathrm{h}}}
\def\DiAl{\Di_\Al}
\def\imp{\rightarrow}

\def\GL{\LogicNamets{GL}}
\def\LC{\LogicNamets{LC}}
\def\Grz{\LogicNamets{Grz}}
\def\DL{\LogicNamets{DL}}
\newcommand\LogicNamets[1]{\logicts{#1}}
\newcommand\logicts[1]{{\textsc{#1}}}
\newcommand\LS[1]{\LogicNamets{S#1}}
\newcommand\LK[1]{\LogicNamets{K#1}}
\def\wK4{\LogicNamets{wK4}}
\def\vL{L}
\def\Alg{\myoper{Alg}}

\def\clA{\mathcal{A}}
\def\clU{\mathcal{U}}
\def\clV{\mathcal{V}}

\def\clC{\mathcal{C}}

\def\clK{\mathcal{K}}
\def\clV{\mathcal{V}}
\def\clS{\mathcal{S}}

\def\clW{\mathcal{W}}
\def\clP{\mathcal{P}}
\def\clQ{\mathcal{Q}}
\def\pwr{\mathbb{P}}
\def\EE{\exists}
\def\AA{\forall}

\def\restr{{\upharpoonright}}

\newcommand\Lind[2]{\myoper{Lind}_{#1}(#2)}

\def\v{\theta}
\def\val{\v}
\def\vext{\bar{\v}}

\newcommand\fragm[2]{{#1}{^{\restr{#2}}}}

\def\vf{\varphi}
\def\mo{\vDash}
\def\vd{\vdash}
\def\con{\wedge}

\def\emp{\varnothing}

\def\se{\subseteq}
\newcommand\languagets[1]{\logicts{#1}}

\def\PV{\Var}
\def\Var{\languagets{Var}}

\def\v{\theta}
\def\vext{\bar{\v}}
\def\iff{\leftrightarrow}
\def\Iff{\Leftrightarrow}

\def\Did{\Di_{\neq}}
\def\Boxd{ {\Box_{\neq}} }

\def\Conv{\noindent {\bf Convention}. }

\newcommand\md[1]{\mathop{\mathrm{md}}{#1}}
\newcommand\mdLog[2]{\mathop{\mathrm{md}_{#2}}{#1}}
\newcommand\mdL[1]{\mdLog{#1}{L}}
\newcommand\h[1]{\mathop{\mathrm{ht}}{#1}}
\newcommand\tra[1]{\mathop{\mathrm{tr}}{#1}}

\newcommand\Ult[1]{\mathop{\mathrm{Ult}}{#1}}

\def\unsymb{\mathrm{u}}
\newcommand\un[1]{{#1}^\unsymb}
\def\tempsymb{\mathrm{t}}
\newcommand\temp[1]{{#1}^{\tempsymb}}

\newcommand\diff[1]{{#1}^{\neq}}

\def\atr{\mathrm{TR}}

\newcommand\clust[1]{\mathop{\mathrm{cl}}#1}

\newcommand\eq[1]{{\equiv_{#1}}}

\newcommand\ff[1]{\bar{#1}}

\begin{document}

\begin{abstract}
It is well known
that for transitive unimodal logics, finite height is
both necessary and sufficient for local tabularity.
It is also well known that for intermediate logics, finite height is
sufficient (but not necessary) for local tabularity.
For non-transitive unimodal, and for polymodal
logics, finite height is necessary (but not sufficient) for local
tabularity.

We discuss generalizations of the
finite height criterion of local tabularity for families of non-transitive and polymodal logics. Then we discuss the finite modal depth
property of modal logics and give a version of the finite height criterion for this property.


\end{abstract}

\maketitle


\section{Introduction}

We study local tabularity of normal propositional (poly)modal logics. Recall that
a logic $\vL$ is {\em locally tabular}, 
if there are only finitely many formulas non-equivalent in $L$ for each finite set of variables. Since logics in our consideration can be viewed as equational theories, local tabularity of a logic is equivalent to local finiteness of the corresponding variety.

\smallskip
In the 1970s, several important results about locally tabular modal and intermediate logics
were obtained.
In the intuitionistic case, there are formulas that
bound the number of elements in every chain in a poset 
\cite{Hosoi1967}.
We call them {\em formulas of finite height}.
There are modal analogs of these formulas that
bound the height of the skeleton of a frame $(X,R)$ with $R$ transitive.
We denote them $B_h$, $h<\omega$.
It turns out that finite height characterizes local tabularity for extensions of
$\LK{4}$, the modal logic of transitive relations.
Namely, in \cite{Seg_Essay} it was shown
that every extension of $\LK{4}$ containing a formula $B_h$ is locally tabular.
It follows from \cite{Maks1975} that there are no other locally tabular logics above $\LK{4}$.
Hence, locally tabular extensions of $\LK{4}$ have an explicit description, both axiomatic and semantic.

It follows from \cite{Seg_Essay} and the Gödel–Tarski translation that intermediate logics of finite height are locally tabular; these results were also obtained independently in \cite{Kuz1973} and \cite{Komori1975Finite}.
Note that in the intuitionistic case, finite height is not necessary for local tabularity: a corresponding example is the Gödel–Dummett logic $\LC$ \cite{Dummett59}.
No axiomatic criterion is known for the case of intermediate logics.
For further results and historical background on local tabularity in the intuitionistic case, see, e.g., \cite{GuramRevaz-2005-Heyt} or a recent paper \cite{Citkin2023},
and references therein.

\improve{looks like}
\improve{(and
as mentioned in \cite{Kom} - by Masahiko Sato; check: \url{https://ems.press/content/serial-article-files/41740})
-terminology? Formulas of finite height?
}
\improve{Mardaev}
\improve{Hosoi - no Kripke frames}
\improve{Maksimova: k4?}

The general picture of non-transitive unimodal, or of polymodal logics
is also unclear. In this case, finite height is necessary, but not sufficient for local tabularity.
It is known that local tabularity implies pretransitivity, some property weaker than transitivity.
 In turn,  pretransitivity allows to express formulas $B^*_h$  that bound the height of the skeleton of a Kripke frame.  In \cite{LocalTab16AiML}, it was shown that every locally tabular logic contains a formula $B_h^*$.
 So pretransitivity and finite height are necessary for all modal logics, including polymodal. However, they do not provide a criterion, as follows from \cite{Byrd78}.
 \improve{
  (to the best of our knowledge, \cite{Byrd78} gives the first example of a pretransitive logic containing a formula of finite height, which is not locally tabular).}

The goals of this paper are to discuss classical results on locally tabular modal logics and their possible generalizations, and to
identify new families of locally tabular  modal logics, in particular
to generalize the finite height criterion for some families of polymodal logics.

\later{\IS{Section 5-  what to say...}
One of these families is described in Theorem \ref{thm:FinH-crit-canon}, which is based on \cite{LocalTab16AiML}.   
Another is described in Theorem \ref{thm:sumAFH} in terms of {\em lexicographic sums of logics}, which are known to preserve local tabularity \cite{LTViaSums}.

Say about clusters explicitly!
}
The paper is organized as follows.
In Section \ref{sec:prel} we remind the main logical and algebraic notions key for our work.
Sections \ref{sec:nec} and \ref{sec:tuned} are primarily overview of classical and some recent results, and a discussion of their   generalizations.
In Section \ref{sec:nec}, we discuss necessary conditions for local tabularity:
finite height; finiteness of a $k$-variable fragment for a fixed $k$; a property of path reducibility. While none of them are sufficient
in general, they imply local tabularity in some narrowed, but interesting, cases.
In particular, we discuss {\em logics that admit the finite height criterion}.
In the unimodal case, this notion means that an extension of the logic is locally tabular iff
it contains a finite height axiom; the polymodal version requires presence of finite height axioms in fragments of the logic.
In Section \ref{sec:tuned}, following \cite{LocalTab16AiML}, we characterize local finiteness and local tabularity in terms of partitions of frames.
In Section \ref{sec:FHC}, we describe some families of polymodal logics which admit the finite height criterion; see Theorems \ref{thm:FinH-crit-canon} and \ref{thm:sumAFH} (these results are based on  \cite{LocalTab16AiML} and  \cite{LTViaSums}).
Section \ref{sec:fmd} is the most technical part of the work.  Here we discuss the {\em finite modal depth} of logics, a property which is at least as strong as local finiteness (but unknown to be stronger).
Many results in this direction were obtained in \cite{Sheht-MD-16}; in particular, it was shown that
above $\LK{4}$, local tabularity and finite modal depth are equivalent.
Our main result in this section is Theorem \ref{thm:md:clusters}, which shows that
in the case of finite height, the finite modal depth is inherited from clusters.
In particular, it
describes a family of logics where finite modal depth is equivalent to finite height, see Corollary \ref{cor:md-cluster-crit}.
It also allows  us to show that local tabularity is equivalent to finite modal depth for a family of non-transitive logics,
and to obtain new upper bounds for modal depth of various logics.

\hide{

Sections 3 and 4 are primarily overview of classical and more recent results on local tabularity in modal logics. In Section 3 we discuss the finite height criterion for transitive logic, a  ., and give a list of necessary conditions for all modal logics. Also, we formulate the notion of ....  for the polymodal case
and observe that ... implies .... ...  We discuss an important result of .....(cluster criterion).
In Section 5, we identify a family of polymodal logics that admit the finite height criterion.
In Section 6 we study a stronger property of {\em  finite modal depth}. We give semantic criterions of this criterion and prove an analog of the cluster criterion  for finite modal depth.

We use it to refine some of earlier obtained upper bounds for modal depth earlier obtained in \cite{Sheht-MD-16},
and also prove the finite modal depth of new logics.

...

Then we discuss the finite modal depth
property of logics and give a version of the finite height criterion for this property.

We list a number of necessary condition of local tabularity:
...
It turns out that in many cases, they are sufficient.

or (for the modal case) in the skeleton.

Namely, in \cite{Seg_Essay} it was shown
that every extension of $\LK{4}$ with a formula $B_h$ of finite height is locally tabular.
In  \cite{Maks1975}, it was shown that there are no other locally tabular logics above $\LK{4}$. \later{check; give refs}.  It follows \cite{Seg_Essay} and the Gödel–Tarski translation that intermediate logics of finite height

It \cite{Seg_Essay}, it was shown that the extensions of $\LK{4}$, the logic of transitive relations,

(a logic is transitive, if it contains the formula $\Di\Di p\imp \Di p$, which expresses the transitivity of a binary relation).

Local tabularity is well-studied for the case of unimodal transitive logics (a logic is transitive, if it contains the formula $\Di\Di p\imp \Di p$, which expresses the transitivity of a binary relation).
According to the classical results by Segerberg \cite{Seg_Essay} and Maksimova \cite{Maks1975},
a transitive logic is locally tabular iff it contains a modal {\em formula $B_h$ of finite height} for some finite $h$.

We study local tabularity of normal propositional (poly)modal logics.

A logic $\vL$ is {\em locally tabular},
if, for every finite $k$, $L$ contains only a finite number of pairwise nonequivalent
formulas in a given $k$ variables.
Equivalently, a logic $\vL$ is locally tabular iff the variety of its algebras is locally finite, i.e., every $L$-algebra is locally finite.
Local tabularity is a strong property:
if a logic is locally tabular, then it has the finite model property (thus it is Kripke complete);
every extension of a locally tabular logic is locally tabular (thus it has the finite model property);
every finitely axiomatizable extension of a locally tabular logic is decidable.

Local tabularity is well-studied for the case of unimodal transitive logics (a logic is transitive, if it contains the formula $\Di\Di p\imp \Di p$, which expresses the transitivity of a binary relation).
According to the classical results by Segerberg \cite{Seg_Essay} and Maksimova \cite{Maks1975},
a transitive logic is locally tabular iff it contains a modal {\em formula $B_h$ of finite height} for some finite $h$. The result of Segerberg

It is well-known that for intermediate logics, finite height is
sufficient (but not necessary) for local tabularity (T. Hosoi, 1967; A.
V. Kuznetsov, 1974). For transitive unimodal logics, finite height is
both necessary and sufficient for local tabularity (K. Segerberg, 1971;
L.L. Maksimova, 1975). For non-transitive unimodal, and for polymodal
logics, finite height is necessary (but not sufficient) for local
tabularity. We will provide a review of classical results on local
tabularity in non-classical logics. Then we will discuss recent
developments in the field, particularly generalizations of the
Segerberg-Maksimova criterion for non-transitive and polymodal cases.

Give main definitions here.  Describe history. Explain the clusters issue. Describe the structure and results.

\IS{Main notions; Local tab; Eq. theories, so it corresponds to the algebraic notion of local finiteness}

}

\improve{continuum on height 3.}
\improve{Explain the clusters issue. }

\section{Terminology and notation}\label{sec:prel}

In this section we remind the main logical and algebraic notions key for our work, and establish the notation used throughout.
For the notions in modal logic, we mainly follow the monographs \cite{CZ}, \cite{BDV}; for the algebraic part,
we refer to \cite{Malcev73}, \cite{BS}, and also to \cite[Appendix B]{BDV}.

Many logical notions considered in this work have purely algebraic counterpart, and vice versa; so in many cases we give dual definitions and statements.

\subsection{Modal language and models}

\subsubsection*{\underline{Language}} Fix a finite set $\Al$, the {\em alphabet of modalities}.
{\em Modal formulas over $\AlA$} are constructed from
a countable set of {\em variables} $\PV=\{p_0,p_1,\ldots\}$ using Boolean connectives 
and unary connectives $\Di\in \AlA$.
The expression $\Box\vf$
denotes $\neg \Di \neg \vf$.
A {\em $k$-formula} is
a formula whose variables are among the set $\{p_i\mid i<k\}$.

\subsubsection*{\underline{Relational models}}

For a set $X$ and a binary relation $R\subseteq X\times X$, we write $aRb$ for $(a,b)\in R$.
We let $R(a)=\{b\mid aR b\}$,
$R[Y]=\bigcup_{a\in Y} R(a)$.

A {\em Kripke frame} is a structure  $\frF=(X,(R_\Di)_{\Di\in \AlA})$,
where $X$ is a set and ${R_\Di\se X{\times}X}$ for $\Di\in \AlA$.
A {\em general frame} is a structure $(X,(R_\Di)_{\Di\in \AlA},\clP)$, where
$\clP$ forms a subalgebra of the powerset Boolean algebra $\pwr(X)$ such that for each $V\in \clV$, and each $\Di\in\Al$,
$R_\Di^{-1}[V]\in \clV$.  The set $X$ is addressed as the {\em domain} of $F$.
\improve{Is it clear? - For a set $X$, $\pwr(X)$ denotes the set of its subsets.}

Let $k\leq\omega$. A {\em $k$-valuation in a frame $F$} is a map from $\{p_i\mid i<k\}$ to $\clP$.
In the case when $F$ is a Kripke frame, $\clP$ is assumed to be the powerset of $X$.
A {\em $k$-model $M$ on} $\frF$ is a pair $(\frF,\theta)$, where $\theta$ is a $k$-valuation.
To define the \emph{truth relation $\mo$} between points in a $k$-model $M$ and $k$-formulas, we interpret Boolean connectives in the usual way, and set
${M,a\mo\Di \vf}$, if ${M,b\mo\vf}$  for some $b$ in $R_\Di(a)$.
We put
$$\vext(\vf)=\{a\mid (F,\v),a\mo\vf\}.$$
A $k$-formula $\vf$ is {\em true in a $k$-model $\mM$}, in symbols $M\mo \vf$, if $\mM,a\mo\vf$ for all $a$ in $\mM$.
A $k$-formula $\vf$ is {\em valid in a frame $\frF$}, in symbols $\frF\mo\vf$,
if $\vf$ is true in every $k$-model on $\frF$.

\subsubsection*{\underline{Algebraic models}}
A {\em modal algebra} $A$ is a~Boolean algebra endowed
with unary operations $g_\Di$,  $\Di\in\Al$,
that satisfy the identities $\Di 0=0$ and  $\Di (x\vee y)=\Di x\vee \Di y$.
A modal formula $\vf$ is {\em valid} in an algebra $A$, if the identity $\vf=1$ holds in $A$.

Every frame defines a modal algebra.  Namely, for a binary relation $R$ on a set $X$, the operation $g:\pwr(X)\to \pwr(X)$ is {\em induced by $R$},
if $g(V)=R^{-1}[V]$ for $V\subseteq X$.
The {\em algebra $\Alg{F}$ of an $\Al$-frame}  $\frF=(X,(R_\Di)_{\Di\in \AlA},\clP)$ is the
Boolean algebra $\clP$ (considered with the standard set operations) endowed with the operations
induced by $R_\Di$, $\Di\in \Al$.
As before, in the case of Kripke frame $(X,(R_\Di)_{\Di\in \AlA})$, we let $\clP$ be the powerset of $X$. So  the algebra $\Alg(X,(R_\Di)_{\Di\in \AlA},\clP)$ is a subalgebra of
the algebra $\Alg(X,(R_\Di)_{\Di\in \AlA})$.
It is immediate that  the algebra of a frame is a modal algebra, and that
in a $k$-model $(\frF,\theta)$,  $\vext(\vf)$  is the value of $\vf$ (considered as an algebraic term in $k$ variables) under $\theta$ (considered as an algebraic assignment) in $\Alg{F}$; details can be found in, e.g,  \cite[Section 5.2]{BDV}.
\improve{dom of g?}

\smallskip
A formula is {\em valid in a class} $\clK$ of (relational or algebraic) structures, if it is  valid in every structure in $\clK$.
Validity of a set of formulas means validity of every formula in this set.

\smallskip

\Conv
When we speak about validity in a class of frames or in a class of algebras, it is always assumed that this class
consists of structures over the same signature.

\smallskip

\subsubsection*{\underline{Representation theorem}}
In fact, every modal algebra is isomorphic to the algebra of a general frame (equivalently, embeds into the algebra of a Kripke frame).
Recall that for a Boolean algebra $A$, the {\em Stone map} embeds $A$ into the powerset algebra $\pwr(\Ult{A})$, where $\Ult{A}$ is the set of ultrafilters of $A$; the map is defined as
$
a \mapsto \{u\in \Ult A\mid a\in u\}.
$
The J\'{o}nsson-Tarski theorem generalizes this result for modal algebras: the Stone map embeds a modal algebra $A$ into the algebra of the Kripke frame $(\Ult{A}, (R_\Di)_{\Di\in\Al})$, where now $\Ult{A}$ is the set of ultrafilters of the underlying Boolean algebra of $A$, and for $\Di\in\Al$, $u,v\in \Ult A$,
$$u R_\Di v  \tiff \AA a\text{ in } A \,(a\in v \Rightarrow  g_\Di(a) \in u).$$
Hence, letting $\clP\subseteq \Ult{A}$ be the Stone image of $A$, we obtain that $A$ is isomorphic to the algebra of the general frame $(\Ult{A}, (R_\Di)_{\Di\in\Al},\clP)$.

\subsection{Modal logics and varieties. Algebraic and Kripke completeness}

\subsubsection*{\underline{Normal logics}}
A  {\em normal } {\em logic} {\em over the alphabet $\Al$} (or just a {\em logic}) is a set $\vL$ of formulas
over $\Al$ that contains all classical tautologies, the formulas $\neg\Di \bot$  and
$\Di  (p_0\vee p_1) \imp \Di  p_0\vee \Di  p_1$ for each $\Di$  in $\AlA$, and is closed under the rules of modus ponens,
substitution and the rule of {\em monotonicity}, which means that for each $\Di$ in $\AlA$, $\vf\imp\psi\in \vL$ implies $\Di \vf\imp\Di \psi\in \vL$.
The smallest logic over the alphabet $\Al$ is denoted by $\LK{}_\Al$. If also $\Phi$ is a set of formulas over $\Al$, we write $\LK{}_{\Al}+\Phi$ to indicate the smallest logic over $\Al$ that contains $\Phi$ (equivalently, contains $\LK{}_{\Al}\cup \Phi$).
An {\em $L$-structure} (a frame or algebra) is a structure where $L$ is valid.

The set of all formulas that are valid in a class $\clK$ of frames or algebras is called the {\em logic of} $\clK$, in symbols: $\Log{\clK}$. For a single structure $B$, we write $\Log{B}$ for $\Log{\{B\}}$.
It is straightforward that  $\Log{\clK}$ is a normal logic.

\Conv
It is a standard convention that the domain of a frame $F$ is non-empty. We do not require this; in particular, the logic of the empty $\Al$-frame is well-defined -- this is just the set of all $\Al$-formulas (the {\em inconsistent logic over $\Al$}), and its algebra is trivial.

\improve{removed:
Moreover, any propositional normal modal logic $L$ is the logic of the class of $L$-algebras, see, e.g.,
\cite[Section 7.5]{CZ}. }

\subsubsection*{\underline{Normal logics as equational theories}}
Normal logics can be equivalently described in terms of identities, valid in modal algberas.
Recall that an {\em equational theory} is the set of all valid identities in a class of algebras.
A {\em variety} is the class of algebras where a given set of identities is valid \cite{Malcev73}.
To relate  normal logics and equational theories,
for a logic $L$  over $\Al$, let $E[L]$ be the set of identities
$\{\vf = \psi \mid \vf \leftrightarrow \psi \in L\}$,
and for the equational theory $E$ of a class of modal algebras over $\Al$, let $L[E]$ be the set of modal formulas $\{\vf \mid \vf = 1 \in E\}$.
It is straightforward that for an equational theory $E$, $L[E]$ is a normal logic, and it is a standard fact
that for a logic $L$, $E[L]$ is an equational theory.
It follows that
modal logics correspond to
varieties of modal algebras.
Another important fact is that the well-known logical construction of the  Lindenbaum algebra $\Lind{L}{k}$ of $L$ over $k$ variables results in the well-known algebraic construction of the free $k$-generated algebra in the variety of $L$-algebras.
See, e.g., \cite[Chapter 7]{CZ} or \cite[Chapter 5]{BDV} for details.
The following standard fact  refines previous observations:
\improve{ Tarski; relate - wording}
\improve{\IS{extensions are undefined}, Hence, modal logics correspond to varieties of modal algebras, and more specifically,
extensions of a logic correspond to subvarieties of the variety of its algebras.}
\begin{proposition}\label{prop:lind} Let $L$ be a logic over $\Al$.
\begin{enumerate}
\item For $k\leq\omega$,   for every $k$-formula over $\Al$, we have:
$\vf\in L\text{ iff }\Lind{L}{k}\mo \vf$. 
\item For every formula over $\Al$,
TFAE:
\begin{enumerate}[(a)]
  \item $\vf\in L$.
  \item For each $k<\omega$, $\Lind{L}{k}\mo \vf$.
  \item $\Lind{L}{\omega}\mo \vf$.
\end{enumerate}
\end{enumerate}
\end{proposition}

\improve{Refs}

\improve{

The {\em $k$-canonical model of} $\vL$ is built from
maximal $\vL$-consistent sets of $k$-formulas; the canonical relations and the valuation are defined in the standard way.  The following fact is well-known, see, e.g., \cite[Chapter 8]{CZ}.
\begin{proposition}\label{prop:k-canonical-model}[Canonical Model Theorem]
Let $\mM$ be the  $k$-canonical model of a logic $\vL$. Then
for all $k$-formulas $\vf$ we have:
$\vL\vd\vf$ iff $\vf$ is true in $M$.
\end{proposition}

}


\subsubsection*{\underline{Kripke completeness and the finite model property}}
It follows from previous facts that every normal logic is {\em algebraically complete}: it is the set of formulas valid in a class of modal algebras.
Due to the representation theorem, every logic is the logic of a class of general frames.
The following property is stronger:
a logic $L$ is {\em Kripke complete}, if $L$ is the logic of a class of Kripke frames.
Even stronger than Kripke completeness property is completeness with respect to finite structures.
Since the Stone embedding of a finite Boolean algebra is an isomorphism, every finite modal algebra is isomorphic to the
modal algebra of a finite Kripke frame. It follows that $L$ is the logic of a class of finite algebras iff it is
the logic of a class of finite Kripke frames. It this case we say that  $L$ has the {\em finite model property}.

\subsubsection*{\underline{Extensions, expansions, and fragments}}
We fix some terminology here. Let $L$ be a  logic over the alphabet $\Al$.

If  $L_1\supseteq L$ and $L_1$ is a logic over the same alphabet $\Al$, then we say that $L_1$ is an {\em extension of $L$}.
By an {\em expansion of $L$} we mean a logic $L_1\supseteq L$, where $L_1$ is considered  over any alphabet of modalities.

For $\AlB\subseteq \AlA$, let $L^{\restr \AlB}$ denote the set of all formulas over $\AlB$ which are in $L$.
\improve{
The logic $L$ is a {\em conservative expansion of a logic $L_0$},  if $L_0=L^{\restr \AlB}$ for some $\AlB\subseteq \AlA$.
(Of course,   the term  {\em conservative extension} is more common, but we reserve {\em extensions} for the logics over the same alphabet.)\IS{we do not use it}.)}

Let 
$B$ be a boolean algebra, $A=(B,(f_\Di)_{\Di\in\AlA})$ a modal algebra.
The algebra $A^{\restr \AlB}=(B,(f_\Di)_{\Di\in\AlB})$ is called the {\em $\AlB$-reduct of $A$}, and $A$ an {\em expansion of} $A^{\restr \AlB}$; for a class $\clK$ of algebras, put $\clK^{\restr \AlB}=\{A^{\restr \AlB}\mid A\in\clK\}$.
The following is standard:
\begin{equation}\label{eq:fragm-alphabet}
\text{$\Log(\clK)^{\restr \AlB}=\Log(\clK^{\restr \AlB})$.}
\end{equation}
In particular, $L^{\restr \AlB}$ is a normal logic.
\later{This is a simplest example of a fragment of $L$. }

Now let $\Psi$ be a set of one-variable formulas over $\Al$.
Let us consider $\Psi$ as a new alphabet of unary connectives.
Formally,  we have the following trivial translation $\vf\mapsto [\vf]$ of modal formulas over $\Psi$ to modal formulas over $\AlA$: it
 is identical on variables, compatible with Boolean connectives, and defined as $[\Di\vf]=\psi([\vf])$ for $\Di=\psi\in\Psi$. Let $\fragm{L}{\Psi}$ be the set of formulas $\vf$ over the alphabet $\Psi$ such that $[\vf]\in L$.
 Note that in general, $\fragm{L}{\Psi}$ is not a logic.
 If $\fragm{L}{\Psi}$ is a logic, then it is called a ({\em normal}) {\em  fragment of $L$},
 and formulas in $\Psi$ are called {\em compound modalities in $L$}.

\begin{example}
If $\AlB\subseteq \AlA$, then $L^{\restr \AlB}$   
can be considered as a simplest example of a fragment of $L$ (let $\Psi=\{\Di p\}_{\Di\in \AlB}$).
 It is not difficult to check that $\Di_1 p\vee \Di_2 p$ or $\Di_1\Di_2 p$ are compound modalities, and so also result in fragments. \improve{It is not difficult to describe
 all fragments of a logic.   }
\end{example}

\improve{not very good}
Semantically, any set of one-variable formulas $\Psi$ over $\Al$ induces Boolean algebras with additional unary operations given by $\Psi$. Formally, for a modal algebra $A=(B,(f_\Di)_{\Di\in\AlA})$, let $A^{\restr \Psi}=(B,(\psi^A)_{\psi\in\Psi})$,
where $\psi^A$ is the interpretation of $\psi$ in $A$.
Trivially, if $\vf$ is a formula over $\Psi$, then
under any assignment in $B$, the value of $\vf$ in $A^{\restr \Psi}$ is the value of $[\vf]$ in $A$.
Note that $A^{\restr \Psi}$ need not be a modal algebra,
but in any case we have for each $\vf$, $\xi$ over the alphabet $\Psi$:
\begin{equation}\label{frag:ident}
  A^{\restr \Psi}\mo \vf=\xi \text{ iff }A\mo[\vf]=[\xi].
\end{equation}
For a class $\clK$ of algebras, let $\clK^{\restr \Psi}=\{A^{\restr \Psi}\mid A\in\clK\}$.

We have the following characterization of fragments:
\begin{proposition}
For a logic $L$ and a set $\Psi$ of one-variable formulas over $\Al$, TFAE:
\begin{enumerate}[(a)]
  \item $L^{\restr \Psi}$ is a fragment of $L$.
  \item  For each $\psi\in\Psi$, $L$ contains
  $\neg\psi(\bot)$ and $\psi(p_0\vee p_1) \leftrightarrow  \psi(p_0)\vee \psi(p_1)$.
  \item For each $L$-algebra $A$,   $A^{\restr \Psi}$ is a modal algebra.
\end{enumerate}
\end{proposition}
\begin{proof}
That (a) implies (b) is immediate: the formulas in (b) 
are translations of the formulas   $\neg\Di\bot$  and
$\Di  (p_0\vee p_1) \leftrightarrow  \Di  p_0\vee \Di  p_1$ with $\Di=\psi$; the latter two formulas are in  $L^{\restr \Psi}$, since $L^{\restr \Psi}$ is a logic.

That (b) implies (c) readily follows from \eqref{frag:ident}.

To prove that (c) implies (a),
note that $L=\Log(\clK)$ for a class $\clK$ of $L$-algebras. Hence:
\begin{equation}\label{eq:compl-for-fragm}
\vf \in L^{\restr \Psi}  \text{ iff } [\vf]\in L \text{ iff } \clK\mo[\vf]=1 \text{ iff } \clK^{\restr \Psi}\mo \vf=1 \text{ iff }
\vf\in \Log(\clK^{\restr \Psi}).
\end{equation}
The first equivalence is the definition of $L^{\restr \Psi}$, the second and the last are by to the definition of the logic of algebras, and the third follows from \eqref{frag:ident}. Thus,
$L^{\restr \Psi}=\Log(\clK^{\restr \Psi})$, so is a logic, and hence is a fragment.
\end{proof}

Also, \eqref{eq:compl-for-fragm} yields the following generalization of
\eqref{eq:fragm-alphabet}: \improve{wording and logic}
\begin{proposition}\label{prop:fragm-compound}~ 
If $L^{\restr \Psi}$ is a fragment of $L$ for a set $\Psi$ of one-variable formulas over $\Al$,
and $L=\Log(\clK)$ for a class $\clK$ of algebras, then $L^{\restr \Psi}=\Log(\clK^{\restr \Psi})$.
\end{proposition}


\subsection{Local tabularity and local finiteness}

\subsubsection*{\underline{Definitions}}

For $k<\omega$, we say that $L$ is {\em $k$-finite},
if there are only finitely many $L$-equivalence classes of $k$-formulas.
In other terms, $k$-finiteness of $L$ means that $\Lind{L}{k}$ is finite.
A logic $\vL$ is {\em locally tabular},
if it is $k$-finite for each $k<\omega$.

\improve{\IS{Reason about fragments: }
If $L$ is $k$-finite, then every fragment of $L$ is $k$-finite as well.
}

\improve{
\IS{If canonical models are envolved:}

Equivalently, $L$ is $k$-finite, if
the $k$-generated free algebra $\Lind{L}{k}$ of $L$ is finite, and equivalently,
the $k$-canonical model of $\vL$ is finite.

The following proposition is immediate.
\begin{proposition} Let $L$ be a logic. TFAE:
\begin{enumerate}[(a)]
\item $L$ is $k$-finite;
\item The $k$-generated free algebra $\Lind{L}{k}$ of $L$ is finite;
\item The $k$-canonical frame of $\vL$ is finite.
\end{enumerate}
\end{proposition}

}

An algebra $B$ is {\em locally finite}, if every finitely generated
subalgebra of $B$ is finite. A class $\clK$ of algebras is {\em locally finite}, if every algebra in $\clK$ is.

\subsubsection*{\underline{Basic observations}}
The following is immediate.
\begin{proposition}\label{prop:k-fin-for-fragm}
 If $L$ is $k$-finite, then each of its fragments is $k$-finite as well.
\end{proposition}
\improve{a remark about non-normal fragments; explain that we are interested in normal only them}
We have the following equivalences.
\begin{proposition}\label{prop:LT-basic-equivs} For a logic $L$, TFAE:
\begin{enumerate}[(a)]
\item $L$ is locally tabular.
\item For each $k<\omega$, $\Lind{L}{k}$ is finite.
\item The algebra $\Lind{L}{\omega}$ is locally finite.
\item The variety of $L$-algebras is locally finite.
\item All fragments of $L$ are locally tabular.
\item All extensions of $L$ are locally tabular.
\item $L$ is a fragment of some locally tabular logic.
\end{enumerate}
\improve{All $k$-canonical models of $L$ are finite.}
\end{proposition}
\begin{proof}
 By Propositions \ref{prop:lind}   and \ref{prop:k-fin-for-fragm}.
\end{proof}
\improve{explain}

It is a standard observation that a formula is valid in an algebra $A$ iff it is valid in every finitely generated subalgebra of $A$.
So we have the following important consequences of local finiteness and local tabularity:
\improve{
 (one direction holds, since taking subalgebras preserve validity of identities; the other direction holds, since there are only finitely many variables in every formula).}
\begin{proposition}~
\begin{enumerate}
\item The logic of a class of locally finite algebras has the finite model property and in particular Kripke complete.
    \item If $L$ is locally tabular, then all extensions of $L$ have the finite model property.
\end{enumerate}

\end{proposition}



\section{Finite height criterion for the unimodal transitive case and general necessary conditions for local tabularity}\label{sec:nec}

\subsection{Criterion}

\improve{
\IS{Old working} In this subsection we discuss the criterion of local tabularity for unimodal logics of transitive frames given
in \cite{Seg_Essay,Maks1975}, and its corollaries for all modal logics \cite{CZ}, \cite{LocalTab16AiML}.
}

Define the {\em height of a poset $\h(X,\leq)$} as $\sup\{|\Sigma|\mid  \text{$\Sigma$ is a finite chain in $(X,\leq)$}\}$. (Of course,  $\h(X,\leq)$ is different from the well-known notion of rank: $\h(X,\leq)$ is defined for all posets, and is not greater than $\omega$.)

\improve{\IS{Old}}
For a binary relation $R$ on a set $X$, let
$R^*$ denote its transitive reflexive closure $\bigcup_{i <\omega} R^i$, where
$R^0$ is the diagonal $Id_X=\{(a,a)\mid a\in X\}$ on $X$, $R^{i+1}=R\circ R^i$, and $\circ$ is the composition.

\improve{\IS{Almost all is old}} For a Kripke frame $\frF=(X,(R_\Di)_{\Di\in \Al})$, put $R_\frF=\bigcup_{\Di\in \Al} R_\Di$. A {\em cluster} in $\frF$ is an equivalence class with respect to the relation $\sim_\frF\; =\{(a,b)\mid a R^*_\frF b \text{ and } b R^*_\frF a\}$. 
For clusters $C, D$, put $C\leq_\frF D$ iff $a R^*_\frF b$ for some $a\in C, b\in D$.
The poset $(X{/}{\sim_\frF},\leq_\frF)$ is called the {\em skeleton of} $\frF$.
The {\em height of a frame} $\frF$, in symbols $\h{\frF}$,
is the height of its skeleton.
For a class $\clF$ of frames, its {\em height $\h{\clF}$} is defined as $\sup\{\h{F}\mid F\in \clF\}$.

\smallskip
In the transitive case, the following formulas  bound the height of a frame $(X,R)$:
\begin{equation}
B_0\;:=\;\bot,  \quad B_{h} \;:=\; p_{h} \to  \Box(\Di p_{h} \lor B_{h-1}).
\end{equation}
Namely, for each $h<\omega$ we have \cite{Seg_Essay}:
\begin{equation}\label{eq:fin-ht-trans}
(X,R)\mo B_h \text{ iff } \h(X,R)\leq h.
\end{equation}

Recall that $\Di\Di p\imp \Di p$ expresses the transitivity of the binary relation in a unimodal Kripke frame.
By a {\em transitive} logic we mean a unimodal logic containing this formula.
 The smallest transitive logic is denoted by $\LK{4}$.

In \cite{Seg_Essay}, it was shown
that every transitive logic containing a formula $B_h$ is locally tabular.
It follows from \cite{Maks1975} that there are no other locally tabular transitive logics.
Hence, we have:
\begin{theorem}\label{thm:thm-seg-maks}
A transitive logic is locally tabular iff it contains $B_h$ for some $h< \omega$.
\end{theorem}


One of the goals of this paper is to discuss possible generalizations of this remarkable fact, as well as its limitations.

\improve{Words}

\subsection{Necessary condition: pretransitivity}

There are non-transitive locally tabular logics. In particular, the logic of every finite frame or algebra is locally tabular due to the following trivial observation:  in a given structure, formulas in $k$-variables correspond to operations of arity $k$, and there are only finitely many of them if the structure is finite (of course, this fact is well-known and pertains to the equational theory of any finite algebraic structure of finite signature).
More examples of non-transitive locally tabular logics will be discussed below.
\improve{Another well-known example of non-transitive locally tabular logic is the unimodal logic
given by the a}

Nevertheless, every locally tabular logic possesses some weaker version of
transitivity.
For a binary relation $R$ and $m<\omega$, put
$R^{\leq m}=  \bigcup_{i \leq m} R^i$. A relation
$R$ is said to be {\em $m$-transitive},  if $R^{\leq m}=R^*$.
It is straightforward that
\begin{equation}
\text{$R$ is $m$-transitive iff
$R^{m+1}\subseteq R^{\leq m}$.  }
\end{equation}
A frame $F$ is said to be  $m$-{\em transitive}, if the relation $R_\frF$ is;  $F$ is {\em pretransitive}, if it is $m$-transitive for some $m<\omega$. A class $\clF$ of frames is {\em pretransitive}, if for some fixed $m<\omega$,
every $F\in\clF$ is $m$-transitive. \improve{Was removed: $R$ is called {\em pretransitive}, if it is $m$-transitive for some finite $m$.}

\improve{example}

As well as transitivity, $m$-transitivity can be expressed by a modal formula.
Let
$\Di^{\leq m} \vf$ denote $\bigvee_{i\leq m} \Di^i \vf$, where $\Di^0 \vf$ is defined as $\vf$, and $\Di^{i+1}\vf$ as $\Di\Di^{i}\vf$.
Let also $\DiAl\vf$ abbreviate $\bigvee_{\Di\in\Al}\Di \vf$.
Consider the formula $\DiAl^{m+1} p \imp \DiAl^{\leq m} p$, which is denoted by $\atr_\Al(m)$.
For any Kripke frame $\frF$, we have \cite[Section 3.4]{KrachtBook}:
\improve{check: and $\Box^{\leq m} \vf$ denote $\neg \Di^{\leq m} \neg \vf$.}
\begin{equation}\label{pretr:sem}
\text{
$R_\frF$ is $m$-transitive iff
$\frF\mo \atr_\Al(m)$. }
\end{equation}
A logic $L$ is said to be {\em $m$-transitive}, if it contains the formula $\atr_\Al(m)$.
The {\em transitivity index $\tra{L}$ of $L$} is
the least $m$ with $\atr_\Al(m)\in L$, if it exists; otherwise, we put $\tra{L}=\omega$.
A logic
$\vL$ is {\em pretransitive}, if $\tra{L}<\omega$, that is $L$ is $m$-transitive for some finite $m$.
We have:
\begin{proposition}\label{prop:pretr}
If a logic is 1-finite, then $\tra{L}<\omega$.
\end{proposition}
This fact is most probably a folklore, see \cite[Exercise 7.15]{CZ}. It can be explained as follows.   The formula $\Di_\Al p $ induces a unimodal fragment $L_0$ of $L$; clearly, $L_0$ inherits 1-finiteness from $L$, so $\Lind{L_0}{1}$ is finite. Due to finiteness, $\Lind{L_0}{1}$  validates the $m$-transitivity formula for some finite $m$  (this is especially easy to see, if represent $\Lind{L_0}{1}$ as the algebra of a finite Kripke frame).
It remains to observe that $m$-transitivity is a 1-formula, and use Proposition \ref{prop:lind}.
\improve{and some more}


\subsection{Necessary condition: finite height}

\smallskip

Assume that a frame $\frF$ is $m$-transitive. In this case,  the
compound modality
$\Di_\AlA^{\leq m}$  relates to $R^*_\frF$. Since
the height of $\frF$  is the height of the preorder $(X,R^*_\frF)$, from
\eqref{eq:fin-ht-trans} we obtain:
\begin{equation}\label{eq:heigth-m}
\text{$\frF\mo B^{\leq m}_h$  iff  the height of $\frF$ is not greater than $h$.}
\end{equation}

\Conv Assume that a logic $L$ is pretransitive, and $\tra{L}=m$. Then we write $\Di^*$ for $\Di_\Al^{\leq m}$.
Also, in this case we use the following notation: for a unimodal formula $\vf$, let $\vf^*$
be the formula obtained from $\vf$ by replacing
each occurrence of $\Di$ with $\Di^*$.
\improve{It is inaccurate.}

\begin{remark}
  The logic with no modalities (that is, the classical proposition logic) is 0-transitive:
  in this case, $\Di_\AlA p$ is $\bot$, so $\atr_\emp(0)$ takes the form of the tautology $p \vee \bot \imp p$.
  \improve{Also, $\Di^*\vf$ is just $\vf$ in this case, and so the height is 1. }
\end{remark}

For a pretransitive logic, its {\em height $\h{L}$} is the least $h$ with
$B_h^* \in L$, if it exists; otherwise,  we put $\h{L}=\omega$.
If $L$ is not pretransitive, its height is also defined as $\omega$.\improve{Zaplata...}

The following is immediate from definitions.
\begin{proposition}
Let $L$ be a logic over $\Al$, and $L_0$ its unimodal $\Di_\Al$-fragment. Then $\tra{L}=\tra{L_0}$ and $\h{L}=\h{L_0}$.
\end{proposition}
\later{... fragment given by the formula...}

The following fact was obtained in \cite{LocalTab16AiML}.
If a unimodal logic is 1-finite, then its height is finite, that is, $L$ contains the formula $B_h^*$ for some $h$. Combining it with Proposition \ref{prop:pretr}, we obtain:
\begin{theorem}\label{thm:1-finite-to-m-h}\cite{LocalTab16AiML}
If a logic $L$ is 1-finite, then
$\h{L}<\omega$.
\end{theorem}
\improve{$\tra{L}<\omega$ and $\h{L}<\omega$.}
\improve{fragments}

Hence, every locally tabular logic $L$ contains a formula of pretransitivity and a finite height formula.
The converse does not hold. There exists a unimodal $L$ with
$\tra{L}=2$ and $\h{L}=1$, which is not locally tabular \cite{Byrd78}, and even not 1-finite \cite{Makinson81}.
We discuss a simple example of such logic below in Example \ref{ex:1-finite-not-two-finite}.
\later{DC: different example? }

Thus in general, finite height is not sufficient for local tabularity. This motivates the following definition.
\begin{definition}\label{def:admFinH-unimodal}
A unimodal logic {\em admits the finite height criterion}, if for each of its
extensions
$L$, we have:
$L$ is locally tabular iff the height of $L$ is finite.
\end{definition}
In this terms, Theorem \ref{thm:thm-seg-maks} says that the smallest transitive logic $\LK{4}$  admits the finite height criterion.
The following generalization is due to \cite{LocalTab16AiML}.
\begin{theorem}\label{thm:LTforLm}
Let $L_m$ be the smallest unimodal logic containing the formula $\Di^{m+1} p\imp \Di p\vee p$. Then for each $m>0$,
   $L_m$ admits the finite height criterion.
\end{theorem}

Now we aim to state the polymodal version of Definition \ref{def:admFinH-unimodal}. In its present form, its condition is far too weak to be interesting in the polymodal case:
\begin{example}
  Let $L$ be the bimodal logic of all frames of the form $(X,R,X\times X)$ (this $L$ is an example  of a  logic with the universal modality, discussed later). Clearly, $L$ is not locally tabular, since one of its unimodal fragments is not. At the same time, all its frames are 1-transitive and of height 1.
\end{example}
Let us therefore make a stronger requirement.
Assuming that a polymodal $L$ over $\Al$ is 1-finite, we also obtain 1-finiteness of all $L^{\restr \AlB}$ with $\AlB\subseteq \Al$ (Proposition \ref{prop:k-fin-for-fragm}). So by Theorem \ref{thm:1-finite-to-m-h}, we have $\tra{L^{\restr \AlB}}<\omega$ and
$\h{L^{\restr \AlB}}<\omega$ for
every such $L^{\restr \AlB}$.

\begin{corollary}\label{cor:finHt-nec}
In $L$ is a 1-finite logic over $\Al$, then $\tra{L^{\restr \AlB}}<\omega$ and
$\h{L^{\restr \AlB}}<\omega$ for
each $\AlB\subseteq \Al$.
\end{corollary}

\begin{definition}\label{def:admFinH-polymodal}
Let $L_0$ be a logic over $\Al$.
Then $L_0$ {\em admits the finite height criterion}, if for each of its extensions $L$, we have:
\begin{center}
$L$ is locally tabular iff
the height of $L^{\restr \AlB}$ is finite for each $\AlB\subseteq \Al$.
\end{center}
\end{definition}
\improve{pretr of $L_0^{\restr \AlB}$  implies pretr of $L^{\restr \AlB}$  }

\hide{
\begin{definition}
Let $L_0$ be a logic over $\Al$ such that $L_0^{\restr \AlB}$ is pretransitive for each $\AlB\subseteq \Al$.
Then $L_0$ {\em admits the finite height criterion}, if for each of its extensions $L$, we have:
\begin{center}
$L$ is locally tabular iff
$\tra{L^{\restr \AlB}}<\omega$ and
$\h{L^{\restr \AlB}}<\omega$ for
each $\AlB\subseteq \Al$.
\end{center}
\end{definition}
}

This definition is more technical than its unimodal analog. But it definitely deserves a consideration: if a polymodal logic
admits the finite height criterion, then its locally tabular extensions have an explicit description, both syntactic and semantic.
We discuss polymodal examples of such logics in Sections \ref{sec:tuned} and \ref{sec:FHC}.
\improve{find good wording somewhere}

\subsection{$k$-finiteness}

Another interesting fact about transitive logics was obtained in \cite{Maksimova89}:
\begin{theorem}\label{thm:Maksimova-1-fin}
If a transitive logic is 1-finite, then it is locally tabular.
\end{theorem}

As well as the finite height criterion, this theorem does not transfer to the general case: a 1-finite logic can be non-locally tabular \cite{Glivenko2021}.

\begin{definition}
A logic {\em admits the $k$-finiteness condition}, if for each of its extensions $L$, we have:
$L$ is locally tabular whenever it is $k$-finite.
\end{definition}
So $\LK{4}$ admits 1-finiteness condition. A simple example of a logic which does not  admit 1-finiteness condition is provided in Example \ref{ex:1-finite-not-two-finite} below.

\smallskip

The following is unknown.
\begin{problem}\label{prob:k-fin}
For a given $\Al$, does the smallest logic $\LK{}_{\Al}$ over $\Al$ admit $k$-finiteness condition for some finite $k$?
At least, is it true for the unimodal case? At least, for the smallest $m$-transitive logic with a given finite $m$?
\end{problem}\improve{logic; least logics?}
\improve{intuitionistic history}

\begin{theorem}\label{thm:fht-imp-1-fin}
If a logic admits the finite height criterion, then  it admits the 1-finiteness condition.
\end{theorem}
\begin{proof}
Assume that $L_0$ admits the finite height criterion, and let $L$ be its 1-finite extension.
We claim that $L$ is locally tabular.

\def\astfrB{{_\ast L^{\restr\AlB}}}
Let $\AlB\subseteq \AlA$. Then the fragment $L^{\restr \AlB}$ is 1-finite.
Hence, this logic is pretransitive by Theorem \ref{thm:1-finite-to-m-h}. Its $\Di^*$-fragment  $\astfrB$ is a transitive 1-finite logic.
By Theorem \ref{thm:Maksimova-1-fin}, this logic is locally tabular. By Theorem \ref{thm:thm-seg-maks}, the finite height criterion for transitive logics,  $\astfrB$ contains a formula $B_h$ for some $h$. So
$L^{\restr \AlB}$ contains $B_h^*$, that is  the height of $L^{\restr \AlB}$ is finite.

Since $L_0$ admits the finite height criterion, $L$ is locally tabular.
\end{proof}
\improve{More details? ref to Shehtman?}

\improve{The converse of this theorem does not hold. \IS{$\LS{5}^2$ and other products. }}

\improve{(nevertheless, recently a wide family of polymodal logics where  1-finiteness imply local finiteness \cite{LTProducts}.) - another text}

\subsection{Necessary condition: reducible path property}

In fact, every locally tabular logic enjoys a property, which is stronger than pretransitivity.
Namely, for $m<\omega$, consider the first-order sentence:
\begin{equation}\label{eq:rppFO}
\AA x_0,\ldots, x_{m+1} \; (x_0Rx_1R\ldots R x_{m+1}\imp \bigvee\limits_{i<j\leq m+1} x_i = x_j
\vee \bigvee\limits_{i< j\leq m}  x_i R x_{j+1}).
\end{equation}
We call it {\em reducible path property}.
A class of unimodal frames is said to be {\em path reducible}, if for some $m$ it validates the above sentence.
It was shown in \cite[Theorem 7.3]{LocalTab16AiML} that the class of Kripke frames of a locally tabular unimodal logic
is path reducible.
Since local tabularity implies Kripke completeness, each unimodal locally tabular logic for some $m$ contains the formula
\begin{equation*}\label{eq:rppMod}
          R_m(\Di): = p_0\con \Di\left(p_1\con \Di\left(p_2\con \ldots \con \Di  p_{m+1}\right)\ldots \right)\imp
          \bigvee\limits_{i<j\leq m+1} \Di^i (p_i \con p_j)
          \vee \bigvee\limits_{i< j\leq m}  \Di^i(p_i \con \Di p_{j+1}),
\end{equation*}
which expresses the property \eqref{eq:rppFO}.

This transfers for the polymodal case immediately. Let $L$ be a locally tabular logic over $\Al$. Then its $\Di_\Al$-fragment is a locally tabular logic, and so contains $R_m(\Di_\Al)$ for some $m$. Moreover,  for each $\AlB\subseteq \AlA$,
the fragment $L^{\restr \AlB}$  of $L$ is locally tabular, and so contains $R_m(\Di_\AlB)$ for some $m$.
So we have the following definition and theorem.  \later{define notation $R_m(\Di_\Al)$.}\improve{This is
a proper way to speak about finite height and pretransitivity as well}
\begin{definition}
A logic $L$ over $\AlA$ is said to be {\em path reducible}, if for every $\AlB\subseteq \AlA$ there exists $m<\omega$ such that $L$ contains $R_m(\Di_\AlB)$.
\end{definition}

\begin{theorem}\cite{LocalTab16AiML}\label{thm:lt-rpp}
 Every locally tabular logic is path reducible.
\end{theorem}

\begin{remark}
 The proof given in \cite{LocalTab16AiML} only needs 2-finiteness of a logic to obtain the reducible path property on the class of its Kripke frames (see the proof of \cite[Proposition 7.4]{LocalTab16AiML}).
 Hence, we have the following corollary:
 \begin{equation}\label{eq:two-fin-forRpp}
 \text{If $L$ is two-finite and Kripke complete, then it is path reducible.}
 \end{equation}
 This fact  implies Theorem \ref{thm:lt-rpp} immediately due to Kripke completeness of locally tabular logics. However, we do not know if
  Kripke completeness  can be omitted in \eqref{eq:two-fin-forRpp}.
\end{remark}

It is straightforward that if a class of unimodal frames is path reducible for $m$, then its frames are $m$-transitive.
Any logic given by a reducible path formula is Kripke complete (reducible path formulas are Sahlqvist), so contains a formula of pretransitivity.

While path reducibility and finite height are both necessary for local tabularity,
they are still not sufficient even for 1-finiteness \cite{ShapSl2024}; see Example \ref{ex:1-finite-not-two-finite} below.

\begin{remark}
In analogy with the previous necessary conditions, we can introduce the following notion.
A logic {\em admits the reducible path criterion}, if for each of its extensions $L$, we have:
$L$ is locally tabular whenever it is path reducible and of finite height. We do not consider this notion in detail in the present paper, and only mention that polymodal logics that admit this criterion
  naturally appear in products of modal logics and close systems, see recent manuscripts \cite{ShapSl2024,Meadors_MS4_Arxiv,ShapSl2025PreLT}.\improve{Give explicit order?}
\improve{The first? And Meadors?}
\end{remark}

\subsection{Counterexamples}
As we discussed, in general finite height is not sufficient for local tabularity \cite{Byrd78}.
Theorem \ref{thm:Maksimova-1-fin} also does not transfer to the general case, since a 1-finite logic can be non-locally tabular. That adding the reducible path property is still not sufficient is also known. The following illustrates these facts with extensions of $\LK{TB}$, the logic of symmetric reflexive frames.

\begin{example}\label{ex:1-finite-not-two-finite}
   Let be the logic $L$ of the frame  $F=(\mathbb{Z},R)$, where
$$aRb \tiff |a-b|\neq 1,$$
and let $L_0$ be the logic of the restriction of $F$ to natural numbers.
These logics have the following `positive' properties:
\begin{enumerate}
  \item $\tra{L}=\tra{L_0}=2$, $\h{L}=\h{L_0}=1$;
  \item $L_0$ and $L$ are path reducible;
  \item $L$ is $1$-finite.
\end{enumerate}
The first two properties follow from immediate semantic arguments. The third is established in \cite[Section 3.4]{LTViaSums}, where it is also shown that neither $L$ nor $L_0$ is locally tabular, and moreover that $L$ is not $2$-finite and $L_0$ is not $1$-finite.

Hence: $L_0$ is an example of a path reducible 2-transitive logic of height 1, which is not 1-finite, and so not locally tabular;
$L$ is an example of a 1-finite, but not 2-finite logic.
\end{example}

\section{Relational characterization of local finiteness and local tabularity}\label{sec:tuned}


While local finiteness is an abstract algebraic notion, in the modal case it enjoys an intuitive visualization:
  since every modal algebra is  embeddable in the algebra of a Kripke frame,  locally finite modal algebras can
be described in terms of special partitions of relational structures. 
This section is based on the method proposed in \cite{LocalTab16AiML}.


\subsection{Local finiteness of powerset modal algebras}
We start with a simple fact about finitely generated Boolean algebras.
For a set  $X$ and  a family $\clP$ of its subsets, define the {\em equivalence $\eq{\clP}$ induced by $\clP$ on $X$}:
$$a\,\eq{\clP}\, b \text{ iff }  \AA P\in\clP \, (a\in P \Leftrightarrow b\in P).$$
Assume that $\clP$ is finite  and consider
the subalgebra $A$ of the powerset Boolean algebra $\pwr(X)$ generated by $\clP$.
One can easily observe that since $\clP$ is finite, every element of $A$ is a union of
$\eq{\clP}$-classes, and the quotient $\clU=X{/}\eq{\clP}$ is the set of atoms of $A$.

Now assume that $\pwr(X)$ is endowed with a unary operation $g$, which distributes over finite unions:   $g:\pwr(X)\to \pwr(X)$, $g(\emp)=\emp$, and $g(U\cup V)=g(U)\cup g(V)$ for all $U,V\subseteq X$.
It is easy to check that $A$ is closed under $g$ iff $g(U)$ belongs to $A$ for each atom $U$ of $A$, that is:
\begin{equation}\label{eq:pretune1}
\text{For every $U\in\clU$,    $g(U)$ is the union of some elements of $\clU$.}
\end{equation}
Equivalently, \eqref{eq:pretune1} can be stated in this form:
\begin{equation}\label{eq:pretune2}
\text{For every $U,V\in\clU$, if $V$ intersects $g(U)$, then $V$ is included in $g(U)$.}
\end{equation}

In fact, \eqref{eq:pretune2} gives a characterization of local finiteness of $(\pwr(X),g)$.
Let $A$ be a subalgebra of $(\pwr(X),g)$ generated by a finite family $\clQ$. If $A$ is finite, then its set of atoms $\clU$ is a finite partition of $X$ which refines $X{/}\eq{\clQ}$ and, as we observed, satisfies \eqref{eq:pretune2}. Conversely, assume that there exists a finite refinement $\clU$ of $X{/}\eq{\clQ}$ that
satisfies \eqref{eq:pretune2}; in this case, the Boolean algebra $B$ generated by $\clU$ in $\pwr(X)$ forms a subalgebra of $(\pwr(X),g)$; moreover, $B$ includes $\clQ$, so $B$ includes $A$, and hence $A$ is finite.

This motivates the following definition and statement.
\begin{definition}\label{def:GenTuned}
A partition $\clU$ of $X$ is said to be {\em $g$-tuned}, if it satisfies \eqref{eq:pretune2}.
$\clU$  is \em{tuned in a modal algebra $C=(\pwr(X), (g_\Di)_{\Di\in\Al})$}, if it is
$g_\Di$-tuned
 for each $\Di\in \Al$.
The algebra $C$ is said to be {\em tunable},
if for every finite partition of $X$ there exists its finite refinement, which is tuned in $C$.
\end{definition}
\begin{proposition}\label{prop:tuned1}~ Let
$C=(\pwr(X), (g_\Di)_{\Di\in\Al})$ be a modal algebra.
\begin{enumerate}
\item
A subalgebra of $(\pwr(X),(g_\Di)_{\Di\in\Al})$ generated by a finite set $\clQ\subseteq \pwr(X)$ is finite iff there exists
a finite partition $\clU$ of $X$ that refines $X{/}{\eq{\clQ}}$ and tuned in $C$.
\item $C$ is locally finite iff  $C$ is tunable.
\end{enumerate}
\end{proposition}

\improve{Example: topological space}

\begin{remark}\label{rem:generalpoweset}
The second statement of Proposition \ref{prop:tuned1}, a criterion of local finiteness of $C$,  can be extended for the case when $C$ is a proper subalgebra of $(\pwr(X),(g_\Di)_{\Di\in\Al})$:
$C$ is locally finite iff
if for every partition of $X$ induced by finitely many elements of $C$, there exists its finite refinement, which is tuned in the powerset algebra $(\pwr(X),(g_\Di)_{\Di\in\Al})$.
\end{remark}

\begin{proposition}
 If the algebra $(\pwr(X),(g_\Di)_{\Di\in\Al})$ is tunable, then its logic has the finite model property.
\end{proposition}
\begin{proof}
The logic of an algebra is the logic of its finitely generated subalgebras, which are finite in our setting.
\end{proof}

While local finiteness of an algebra guarantees the finite model property of its logic $L$, it does not imply local tabularity of $L$ (in other terms, the variety generated by a locally finite algebra is generated by its finite members, but need not be locally finite).  The  property of local tabularity can be ensured by the following criterion given by A. Maltsev \cite{Malcev73}.

\smallskip
A class  $\clA$ of algebras of a finite signature is said to be
{\em uniformly locally finite}, if
there exists a function $f:\omega\to \omega$ such that the cardinality of a subalgebra of any $B\in \clA$ generated by
$k<\omega$ elements does not exceed $f(k)$.\improve{copy}
\begin{theorem}\cite[Section 14, Theorem 3]{Malcev73}\label{Malcev73}.
Local finiteness of the variety 
generated by a class $\clA$ of algebras 
is equivalent to uniform local finiteness of $\clA$.
\end{theorem}

Clearly, for algebras on Boolean base, uniform local finiteness can be given via upper bound on the number of  atoms. In our terms, this pertains to the size of tuned partitions. Hence, we get our next definition and criterion.
\begin{definition}\label{def:GenTuned:uniform}
A class $\clA$ of powerset $\Al$-algebras is said to be
{\em $f$-tunable} for a function $f:\omega\to\omega$,
if for every $C=(\pwr(X),(g_\Di)_{\Di\in\Al})\in\clA$, for every finite partition $\clV$ of $X$ there exists a refinement
$\clU$ of $\clV$ such that $|\clU|\leq f(|\clV|)$ and $\clU$ is tuned in $C$.
The class $\clA$ is \em{uniformly tunable}, if it is $f$-tunable for some $f:\omega\to\omega$.
\end{definition}
\begin{theorem}\label{thm:LTforPowerset}
For a class $\clA$ of powerset $\Al$-algebras, $\Log(\clA)$ is locally tabular iff $\clA$ is uniformly tunable.
\end{theorem}

\subsection{Locally finite algebras of frames}

Now we consider the case when the modal operation $g$ on $\pwr(X)$ is induced by a binary relation $R$ on $X$, that is $g(V)=R^{-1}[V]$ for $V\subseteq X$.
In this case, the condition \eqref{eq:pretune2} takes the following form:
\begin{equation}\label{eq:part}
\text{$\AA\, U,V\in \clU$\,} (\EE a\in U  \EE b\in V \, (aR b)  \,\Rightarrow\, \AA a\in U  \EE b\in V \, (aR b)).
\end{equation}
\begin{definition}\label{def:tuned}
A partition $\clU$ of $X$ is said to be {\em $R$-tuned} for a binary relation $R$ on $X$, if $\clU$ satisfies \eqref{eq:part}.
A partition $\clU$ of $X$ is said to be {\em tuned in a frame $\frF=(X,(R_\Di)_{\Di\in \Al})$}, if it is $R_\Di$-tuned for every $\Di\in \Al$.
\end{definition}
To the best of our knowledge, tuned partitions of frames were initially considered
by H{\aa}kan Franz\'{e}n in his semantic proof of Bull's theorem, see \cite{Franz-Bull}.

\smallskip
The notions and facts from the previous subsection transfer immediately
for the case of frames: by the tunability of a frame we  mean the tunability of its algebra.
However, the concretization \eqref{eq:part} of a more abstract \eqref{eq:pretune2} turns out to be a very convenient tool for analyzing local finiteness of modal algebras and local tabularity of modal logics. So we consider it in details.


\begin{proposition}\label{prop:tuned-in-forms}
Let $\clU$ be a finite partition of a frame $F=(X,(R_\Di)_{\Di\in \Al})$, $\sim$ the corresponding equivalence, that is $\clU\,=\,X{/}{\sim}$.
TFAE:
\begin{enumerate}[(a)]
\item\label{prop:tuned-a} $\clU$ is tuned.
\item\label{prop:tuned-b} Unions of sets in $\clU$ form a subalgebra of the modal algebra $\Alg(F)$.
\hide{
\item The Boolean algebra generated by $\clU$ in $\pwr(X)$ is closed under the operations $R_\Di^{-1}[~]$,
and hence form a modal subalgebra of $\Alg(F)$.
}
\item\label{prop:tuned-c} The canonical projection $a\mapsto [a]_\sim$ is a p-morphism from $F$ onto the frame $F_\sim$. Here $F_\sim$ is the frame $(\clU,(\ff{R}_\Di)_\Al)$, where
    $[a]_\sim\ff{R}_\Di  [b]_\sim$ iff $a'R_\Di b'$ for some $a'\sim a$ and $b'\sim b$.
\item\label{prop:tuned-d} For every $U,V\in \clV$ and $\Di\in \Al$, we have:
$$
\text{$U\subseteq R_\Di^{-1}[V]$ or $U\cap R_\Di^{-1}[V]=\emp$}.
$$
\improve{
\item\label{prop:tuned-e} For any $k$-model $M=(F,\val)$, if there is a finite set $\Psi=\{\psi_i\}_{i<k}$ of $k$-formulas such that
$\clU=\{\vext(\psi)\mid \psi\in \Psi\}$, then every $k$-formula is $M$ equivalent to
an elementary conjunction of formulas in $\Psi$.}
\item\label{prop:tuned-f} ${\sim}\circ R_\Di\;\subseteq\; R_\Di\circ{\sim}$ for each $\Di\in \Al$. \improve{Bisim from  slides?}
\end{enumerate}
\end{proposition}
The equivalence \eqref{prop:tuned-a} $\Iff$ \eqref{prop:tuned-c} is due to Franz\'{e}n, see \cite[Lemma 2]{Franz-Bull}.
Its consequence, the equivalence \eqref{prop:tuned-a} $\Iff$ \eqref{prop:tuned-b}, is explicitly stated in
\cite[Corollary 3.3]{Blok1980}. The condition  \eqref{prop:tuned-d} is a form of  \eqref{eq:pretune2} for relation-induced modal operations, and \eqref{prop:tuned-f} is  a laconic form of \eqref{eq:part}.

\improve{

\begin{example}
\IS{ natural.}
\end{example}
}

A characterization of locally tabular logics is given by
\begin{corollary}\label{cor:tuned-for-frames}~
\begin{enumerate}
\item The logic of a class  $\clF$ of Kripke frames is locally tabular
iff $\clF$  is uniformly tunable.
\item A logic $L$ is locally tabular iff $L$ is the logic of a uniformly tunable class of Kripke frames.
\end{enumerate}
\end{corollary}
\begin{proof}
The first statement is a reformulation of Theorem \ref{thm:LTforPowerset} for relation-induced modal operations.
The second statement follows from the fact that any locally tabular logic is Kripke complete.
\end{proof}

\begin{remark} This corollary was first formulated in \cite{LocalTab16AiML}, where its proof was given in terms of Kripke models (the proof was given for the unimodal case; the polymodal version of this proof is given in \cite{LTViaSums}).
\later{To the best of our knowledge, Theorem \ref{thm:LTforPowerset} has never been formulated before.\IS{Why these details?}}
\end{remark}

\medskip

Subalgebras of modal algebras of Kripke frames characterize any modal logic, so we explicitly restate the results of the previous subsection for them as well. We can easily adapt our notions for this case,
similarly to the reasoning provided in Remark \ref{rem:generalpoweset}.
\begin{definition}\label{def:tunable-generalFrame}
A general frame $F=(X,(R_\Di)_{\Di\in\Al},\clP)$ is {\em tunable},
if for every finite partition $\clV$ of $X$ with
\begin{equation}\label{eq:def-restr-on-general}
\clV\subseteq \clP
\end{equation}
 there exists its finite refinement, which is $R_\Di$-tuned for each $\Di\in \Al$.
If, in a class  $\clF$ of general frames, for some fixed $f:\omega\to\omega$,  for every $F\in\clF$, the size
of a tuned refinement $\clU$ of a finite partition $\clV$ of $F$ can be bounded by $f(|\clV|)$, then $\clF$ is \em{uniformly tunable}.
\end{definition}
So the only difference between the tunability in Kripke and general frames is the extra condition \eqref{eq:def-restr-on-general} for the latter.
In particular, we have:
\begin{corollary}\label{cor:tunable-frame}
The algebra of a general frame $F$ is locally finite iff $\frF$ is tunable.
\end{corollary}

This gives a criterion of local tabularity for the case of general frames:
\begin{corollary}\label{cor:tuned-for-Genframes}~
The logic of a class  $\clF$ of general frames is locally tabular
iff $\clF$  is uniformly tunable.
\end{corollary}

\hide{
In terms of logics, we have:

\begin{corollary}\label{cor:tunable}
 If $L=\Log(\clF)$ for a uniformly tunable class of general frames, then $L$ is locally
 tabular.
\end{corollary}
}


We illustrate this discussion with two simple examples, where we consider two types of expansions: with the universal and with the difference modalities.

\begin{example}\label{eq:universalExpLT}
For a logic $L$, let $\un{L}$
be the expansion of $L$ with the universal modality \cite{GorankoPassy1992}.

It is known that
\begin{equation}\label{eq:ltU}
\text{$\un{L}$ is locally tabular iff $L$ is locally tabular.}
\end{equation}
The `only if' direction is trivial: $L$ is a fragment of $\un{L}$.
The `if' direction is also simple; we provide a syntactic and a semantic argument for it.

The syntactic argument follows from a well-known fact that every formula in $\un{L}$ is equivalent to a Boolean combination of formulas, 
where the universal modality can only appear once at the beginning of the formula \cite[Section 3.2]{GorankoPassy1992}; then one can use a well-known fact about local tabularity of the logic $\LS{5}$ of universal relations. 

A semantic argument for \eqref{eq:ltU} is given in \cite{Tacl2017}: trivially, every partition is tuned with respect to the universal relation.

 \later{Read Scrogg (double check)}
 \improve{check Halmos.\later{Maltcev - DC}


@article{halmos1956algebraic,
  author    = {Halmos, Paul R.},
  title     = {Algebraic Logic {I}: Monadic Algebras},
  journal   = {Compositio Mathematica},
  year      = {1956},
  volume    = {12},
  pages     = {217--249}
}

}

The following is immediate from \eqref{eq:ltU} and the finite height criterion for transitive logics \cite{Seg_Essay,Maks1975} (Theorem \ref{thm:thm-seg-maks}): for every extension $L$ of $\un{\LK{4}}$,
 \begin{equation}\label{eq:finh-U}
  \text{$L$ is locally tabular iff its $\Di$-fragment contains a formula of finite height.}
 \end{equation}
 Indeed, consider the $\Di$-fragment $L_0$ of $L$ and observe that $\un{L_0}\subseteq L$.\footnote{
 In \cite{Bezh-Mead2024},
 \eqref{eq:finh-U} is proved for extensions of  $\un{\LS{4}}$ by
 algebraic methods.}
 \later{\IS{include later}
 \footnote{
 In \cite{Bezh-Mead2024}, a particular case of \eqref{eq:finh-U} was rediscovered for extensions of  $\un{\LS{4}}$.}
 }

 Due to \cite{LocalTab16AiML} (Theorem \ref{thm:LTforLm}),  \eqref{eq:finh-U}  generalizes for extensions of
 $\un{L_m}$, where $L_m$ is defined by the formula $\Di^{m+1} p\imp \Di p\vee p$.
  In our terms, this means that every logic
  $\un{L_m}$ admits the finite height criterion: clearly, the universal modality fragment  has height 1.
 In general, \eqref{eq:ltU} takes the following form for every logic:
\begin{equation}\label{eq:finh-U-adm}
\text{$\un{L}$ admits the finite height criterion iff $L$  admits the finite height criterion.}
\end{equation}
\end{example}

\smallskip

\begin{example}\label{eq:diffExpLT}
Let $L$ be a logic over $\Al$, and let $\Al_1$ be the alphabet enriched with a new symbol $\neq$. This symbol will be interpreted as the
{\em difference modality} \cite{DeRijkeDiff}. Namely,
the logic $\diff{L}$ is defined as  the smallest logic over $\Al_1$ that contains $L$ and the following formulas:
\begin{equation}\label{eq:diff-axioms}
p\imp \Boxd\Did  p, \quad \Did \Did p\imp \Did p\vee p, \quad
\Di p\imp \Did p\vee p\text{ for each $\Di\in\Al$}.
\end{equation}
It is known that $\diff{L}$ is characterized by general frames $\frF=(X,(R_\Di)_{\Di\in \Al_1},\clP)$ such that the $\Al$-reduct of $F$ is an $L$-frame
and
\begin{equation}\label{eq:diffRel}
 R_{\neq} \cup  Id_X\; = \; X\times X.
\end{equation}
\improve{more refs; more words for the completeness; explain that this is a fragment}
\improve{Say: take point-generated subframes ..}

It easily follows from \eqref{eq:diffRel} that every partition of $X$ is tuned with respect to $R_{\neq}$; see \cite[Proposition 5.5]{LocalTab16AiML} for details.
From Corollary  \ref{cor:tuned-for-Genframes}, we obtain
the following 
generalization of \eqref{eq:finh-U-adm}:
\begin{equation}\label{eq:finh-Diff}
\text{$\diff{L}$ admits the finite height criterion iff $L$  admits the finite height criterion.}
\end{equation}
In particular, $\diff{L_m}$ are examples of logics with the difference modality that admit the finite height criterion.
\end{example}

\begin{remark}
While Example \ref{eq:universalExpLT} has been known before,
to the best of our knowledge
Example \ref{eq:diffExpLT} is new.
\end{remark}
\improve{Include? - (in particular, \eqref{} was explicitly formulated in \cite{Tacl2017})}

The previous examples show that in some cases, combining two locally tabular logics results in a locally tabular one. This is not the case in general. The simplest example is the fusion of two instances of a locally tabular logic $\LS{5}$: this logic is not even pretransitive.
\later{Below we discuss constructions  of polymodal logics that admit the finite height criterion.}

\improve{
Another example is the {\em temporal expansion.}
\begin{example}
$\temp{L}$. \IS{Wolter}; then transfer the fmp.
\end{example}
}

\improve{
\IS{change this text}
In the next section we discuss a subtler form of the tunability criterion and describe more polymodal logics that admit the finite height criterion.
}

\section{Polymodal logics that admit the finite height criterion}\label{sec:FHC}
In this section we address the following question: what additional conditions on a unimodal non-transitive logic, or on a polymodal logic, imply the finite height criterion of local tabularity?
The main results of this section are given in  Theorems \ref{thm:FinH-crit-canon} and \ref{thm:sumAFH},
and are based on the following  technical tools: the {\em cluster criterion} \cite{LocalTab16AiML}, which is discussed in the next subsection (Theorem \ref{thm:clust-crit}),  and the {\em lexicographic sum of logics}, discussed in Section \ref{sec:lex}.

\subsection{Finite height criterion via clusters}

\improve{
[New: subframe (or subinterval) logic $L$ is lt iff $L[1]$ is.]
}
\improve{Use logic from slides}

For a Kripke  frame $\frF=(X,(R_\Di)_{\Di\in \Al})$, the {\em restriction  $\frF\restr Y$  of $\frF$ to its subset $Y$} is the frame $(Y, (R_\Di\cap (Y\times Y))_{\Di\in\Al})$.

A frame $F$ is a {\em cluster}, if $a R_F^* b$ for all $a,b$ in $F$.
For a class $\clF$ of Kripke frames, $\clust{\clF}$ is the class of cluster-frames which are restrictions on clusters (as sets) occurring in frames in $\clF$:
$$
\clust{\clF}=\{F\restr C\mid C \text{ is a cluster in } F\in\clF\}.
$$

The following characterization of locally tabular logics was obtained in \cite{LocalTab16AiML}.
\begin{theorem}\label{thm:clust-crit}
The logic of a class $\clF$ of Kripke frames is locally tabular iff
the height of $\clF$ is finite and $\Log\clust{\clF}$ is locally tabular.
\end{theorem}
\improve{
\IS{Mention or not?}
\begin{corollary}
A logic $L$ is locally tabular iff $L$ is the logic of a class $\clF$ such that
the height of $\clF$ is finite and $\Log\clust{\clF}$ is locally tabular.
\end{corollary}
}
\noindent
In \cite{LocalTab16AiML}, this theorem is stated for the unimodal case. The proof is based on the construction of
tuned partitions in a frame of finite height, starting from tuned partitions in its clusters.
The polymodal version of this theorem is based on the same idea, details are given in \cite{ShapSl2024}. \improve{The idea of the proof is the following:....}

For a  frame $\frF=(X,(R_\Di)_{\Di\in \AlA})$ and $\AlB\subseteq \AlA$,
let $\fragm{F}{\AlB}=(X,(R_\Di)_{\Di\in \AlB})$.

For a logic $L$, let $F_L$ be the representation Kripke frame of $\Lind{L}{\omega}$, the {\em canonical frame of $L$}.
Recall that a logic $L$ is {\em canonical}, if it is valid in $F_L$.
In \cite[Theorem 6.4]{Seg_Essay}, it is proved that if a unimodal transitive logic $L$ contains a formula $B_h$ of finite height, then
$\h{F_L}\leq h$. It is not difficult to
 generalize this proof to obtain the following proposition:
 \begin{proposition}\label{prop:BhCanon}
  If $L$ is a logic over $\Al$, $\AlB\subseteq \AlA$, and $\h{\fragm{L}{\AlB}}=h<\omega$. Then
$\h{\fragm{F_L}{\AlB}}=h$.
 \end{proposition}
 \improve{better ref for Seg}

Let $\clF$ be the class of all frames of a logic $L$. We say that $L$ has the {\em ripe cluster property},  if $\clust{\clF}$ is uniformly tunable, or equivalently, $\Log{\clust{\clF}}$ is locally tabular.

The following generalizes \cite[Theorem 5.13]{LocalTab16AiML} for the polymodal case.

\begin{theorem}\label{thm:FinH-crit-canon}
Suppose $\vL_0$ is a canonical pretransitive logic with the ripe cluster property.
Then for every extension $\vL$ of $\vL_0$, we have:
\begin{center}
$\vL$ is locally tabular iff $\vL$ is of finite height.
\end{center}
\end{theorem}
\begin{proof}
The `only if' direction follows from Theorem \ref{thm:1-finite-to-m-h}.

\smallskip

Let $h=\h{L}<\omega$.
Then $L$ contains the formula $B_h^*$. Let $L'$ be the smallest logic that contains $L_0$ and $B_h^*$.
By Proposition \ref{prop:BhCanon}, $\vL'$ is canonical. Hence, $\vL'$ is Kripke complete; since it extends $L_0$,  it has the ripe cluster property. So  $\vL'$ is locally tabular by Theorem \ref{thm:clust-crit}.

Clearly, $L$ contains  $L'$, so $L$ is locally tabular too.
\end{proof}

\begin{corollary}
    Every canonical pretransitive logic with the ripe cluster property
  admits   the finite height criterion.
\end{corollary}
\begin{remark}
  In fact, such logics admit some simpler version of the finite height criterion, since we only require the
finite height of a logic, and do not consider its fragments.
\end{remark}
Due to Theorem \ref{thm:fht-imp-1-fin}, we have:
\begin{corollary}
  Every canonical pretransitive logic with the ripe cluster property
  admits the 1-finiteness condition.
\end{corollary}

\improve{

\begin{corollary}
 \IS{Admits  FHC; Not corollary? Look at fragments! Universal; Difference;}
\end{corollary}

\begin{remark}
  Say that the condition is stronger (weaker)
  \IS{Stronger}
\end{remark}

}




\subsection{Finite height criterion via lexicographic sum}\label{sec:lex}
In this section, we consider the operation of lexicographic sum of Kripke frames and of logics.
Unlike many other operations on logics,  lexicographic sum
preserves local tabularity \cite{LTViaSums}. We use it to identify new polymodal logics that admit the finite height criterion.

\def\OmA{\Al_\mathrm{v}}
\def\OmB{\Al_\mathrm{h}}


\improve{
where a family of frames-summands is  indexed by elements of another frame, and the corresponding operation on logics.wording}

Throughout this section, we assume that $\Al_1$ and $\Al_2$ are two disjoint finite sets,
and $L_1$ and $L_2$ are logics over
$\Al_1$ and $\Al_2$, respectively.
Elements of $\Al_1$ and $\Al_2$ are called {\em vertical} and {\em horizontal modalities}, respectively.


\begin{definition}\label{def:sumLex}
Let $\frI=(Y,(S_\Di)_{\Di\in \Al_1})$ be a frame, and let $(\frF_i)_{i\in Y}$ be a family of frames such that
  $\frF_i=(X_i,(R_{i,\Di})_{\Di\in \Al_2})$.
The {\em lexicographic sum} $\LSuml{\frI}{\frF_i}$ is the $(\Al_1\cup\Al_2)$-frame
$\left(\bigsqcup_{i \in Y} X_i, (S^\oplus_\Di)_{\Di\in\Al_1}, (R_\Di)_{\Di\in\Al_2}\right)$,  where
$\bigsqcup_{i\in Y}{X_i}=\bigcup_{i\in Y}(\{i\}\times X_i)$, and:
\begin{eqnarray*}
\text{for }\Di\in\Al_1,&& (i,a) S^\oplus_\Di  (j,b) \text{ iff }   i S j;\\
\text{for }\Di\in\Al_2,&& (i,a) R_\Di  (j,b)  \text{ iff }    i = j \;\&\;a R_{i,\Di}  b.
\end{eqnarray*}
For a class  $\clI$ of $\Al_1$-frames and a class  $\clF$ of $\Al_2$-frames, $\LSuml{\clI}{\clF}$ denotes
 the class of all sums
$\LSuml{\frI}{\frF_i}$, where $\frI \in \clI$ and all $\frF_i$ are in $\clF$.
For modal logics $L_1$ and $L_2$,
let $\LSuml{L_1}{L_2}$ be the logic of the class
$\LSuml{\clF_1}{\clF_2}$, where $\clF_i$ is the class of all frames of the logic $L_i$.
\end{definition}

\begin{remark}
In we additionally require that all summands are equal, then this operation results in the {\em lexicographic
product of logics} introduced in \cite{Balb2009}.
\end{remark}

\begin{theorem}\cite{LTViaSums}\label{thm:localfin-lex-semantically}
If the logics $L_1$ and $L_2$ are locally tabular, then
the logic $\LSuml{L_1}{L_2}$  is locally tabular as well.
\end{theorem}


While in general a complete axiomatization of $\LSuml{L_1}{L_2}$ is unknown, 
in many cases it can be obtained in the following way.
\begin{definition}
Let
$\Phi(\Al_1,\Al_2)$ be the set of all formulas
\begin{equation}\label{eq:alpha-beta-gamma}
\Dih \Div p\to \Div p, \; \Div\Dih p\to \Div p, \;\Div p\to \Boxh \Div p
\end{equation}
with $\Div$ in $\Al_1$ and $\Dih$ in $\Al_2$.
Define $L_1\oplus L_2$ as the smallest logic over $\Al_1\cup\Al_2$ that contains $L_1\cup L_2\cup\Phi(\Al_1,\Al_2)$
\end{definition}
 The formulas $\Phi(\Al_1,\Al_2)$ were considered in \cite{Balb2009} in connection with axiomatization problems
of lexicographic products,
  and also in \cite{Bekl-Jap}
in the context of polymodal provability logic.
It is straightforward that formulas $\Phi(\Al_1,\Al_2)$ are valid in any lexicographic sum.
Moreover,
\begin{equation}\label{eq:lexAx}
  \LSuml{L_1}{L_2}=L_1\oplus L_2
\end{equation}
holds for many logics.
For instance, it follows from  \cite{Balb2009} that
$\LSuml{\LK{4}}{\LK{4}}=\LK{4}\oplus \LK{4}$;
\eqref{eq:lexAx}  is also true for the sum $\LSuml{\GL}{\GL}$, where $\GL$ is the   G\"odel-L\"ob logic \cite{Bekl-Jap}.

However, in general the operation $\LSuml{L_1}{L_2}$ does not preserve $L_1$,
that is $L_1$ is not included in $\LSuml{L_1}{L_2}$;
clearly, in this case
\eqref{eq:lexAx} does not hold.
Another obstacle is that
$\LSuml{L_1}{L_2}$ is Kripke complete, which
is not guaranteed for $L_1\oplus L_2$ (to the best of our knowledge, even in the case of Kripke complete $L_1$ and $L_2$).
Nevertheless, it can be shown that \eqref{eq:lexAx} holds for a broad family of sums, where $L_1$ is axiomatizable by universal Horn modal formulas.\improve{closed formulas, PTC formulas etc}

\medskip
 \improve{Lemma 5.9}

The following two facts are given in \cite{LTViaSums}.
\begin{proposition}\label{lem:sum-is-contained-inPsi}
If $L_1\oplus L_2$ is Kripke complete, then $\LSuml{L_1}{L_2}\subseteq L_1\oplus L_2$.
\end{proposition}
\begin{theorem}\label{thm:sum-of-canon-LT}
If $L_1$ and $L_2$ are canonical locally tabular logics, then $L_1\oplus L_2$ is locally tabular.
\end{theorem}
\begin{remark}
The proof of this theorem uses Kripke completeness of $L_1\oplus L_2$.
If $L_1$ and $L_2$ are canonical, then $L_1\oplus L_2$ is canonical as well, since the formulas
$\Phi(\Al_1,\Al_2)$ are Sahlqvist.
However,
there are non-canonical locally tabular logics \cite[Section 6]{Goldblatt1995}.
We do not know if the canonicity can be omitted in Theorem \ref{thm:sum-of-canon-LT}:
does Kripke completeness of $L_1\oplus L_2$ follow from local tabularity of  $L_1$ and $L_2$?
\end{remark}

We use Theorem \ref{thm:sum-of-canon-LT} to describe the following families of logics admitting the finite height criterion.
\later{(this result was announced in \cite{LogCol2023}).
\improve{Do we need it? Is it correct?}}
\begin{theorem}\label{thm:sumAFH}
 Let $L_1$ and $L_2$ be canonical logics.
 Assume that $L_1$ and $L_2$ admit the finite height criterion.
 Then $L_1\oplus L_2$ admits the finite height criterion.
\end{theorem}
\improve{What about $\LSuml{L_1}{L_2}$?}
\begin{proof}\improve{other direction}
Consider an extension $L$ of $L_1\oplus L_2$.

\smallskip
If $L$ is locally tabular, then the height of $\fragm{L}{\Al}$ is finite for each
$\Al\subseteq \Al_1\cup\Al_2$  by Corollary  \ref{cor:finHt-nec}.

\smallskip
Assume that $\h{\fragm{L}{\Al}}<\omega$ for each
$\Al\subseteq \Al_1\cup\Al_2$ and show that $L$ is locally tabular.

Let $\psi_\Al$ denote the corresponding formula of finite height of the logic $\fragm{L}{\Al}$. Now put $\Gamma_1=\{\psi_\Al\mid \Al\subseteq \Al_1\}$,
$\Gamma_2=\{\psi_\Al\mid \Al\subseteq \Al_2\}$. Let $L_1'$ be the smallest logic over $\Al_1$ that
contains $L_1\cup \Gamma_1$, and let $L_2'$ be the smallest logic over $\Al_2$ that
contains $L_2\cup \Gamma_2$.
Since  $L_1$ and $L_2$ admit the finite height criterion, the logics  $L_1'$  and $L_2'$ are locally tabular.
\improve{use notation from prel, or remove it from there}

Since $L_1$ and $L_2$ are canonical,  $L_1'$ and $L_2'$ are canonical as well by Proposition \ref{prop:BhCanon}.
By Theorem \ref{thm:sum-of-canon-LT}, $L_1'\oplus L_2'$ is locally tabular.
\improve{statement; refs}

Clearly, $L$ includes the logics $L_1'$ and $L_2'$; since $L$ extends $L_1\oplus L_2$, it also includes $\Phi(\Al_1,\Al_2)$.
Hence, $L$ is an extension of the locally tabular logic $L_1'\oplus L_2'$,
and so is locally tabular too.
\end{proof}

\section{Logics of finite modal depth}\label{sec:fmd}
In this section we discuss the finite modal depth property of logics, which is at least as strong as local finiteness.
Many results in this directions were obtained by V. Shehtman in \cite{Sheht-MD-16}. In particular, finite modal depth was shown for all locally tabular logics above $\LK{4}$, for the difference logic and other examples of non-transitive logics \cite{Sheht-MD-16}.
\improve{These results are based
on the method of bisimulation games. }

Our main result is Theorem \ref{thm:md:clusters}, an analog of Theorem \ref{thm:clust-crit} for finite modal depth: it shows that
in the case of finite height, the finite modal depth is inherited from clusters. In particular, it
describes a family of logics where finite modal depth is equivalent to finite height, see Corollary \ref{cor:md-cluster-crit}. It also allows  us to show that local tabularity is equivalent to finite modal depth for all 1-transitive unimodal logics,
and obtain new modal depth upper bounds.
\later{Last minute change: general frames became Kripke frames. Old version: v10}

\subsection{Background: finite model depth in models}
We start with the exposition of some facts that follow from \cite{Sheht-MD-16}.

\begin{definition}
For a formula $\vf$, its {\em modal depth $\md{\vf}$} is the maximal number of nested modalities occurring in~$\vf$. Let $L$ be a logic over $\Al$. The {\em $L$-modal depth $\mdL{\vf}$ of $\vf$} is $\min\{\md{\psi}\mid \vf\iff \psi \in L\}$. The {\em modal depth $\md{L}$ of $L$}
is
$\sup\{\mdL{\vf}\mid \vf \text{ is a formula over $\Al$}\}$.
\improve{Check Kracht; $\leq d$}
\end{definition}

\begin{example}
The logic $\LS{5}$ is a well-known example: it is easy to see that $\md{\LS{5}}=1$.
\end{example}

\begin{example}
Let $\DL$ be the difference logic, the logic of frames of the form $(X,\neq)$.  Recall that $\DL$ is defined by the formulas
$p\imp \Box\Di p$,  $\Di \Di p\imp \Di p\vee p$.
Due to \cite{Sheht-MD-16}, $\md{\DL}=2$.
\end{example}

\improve{
\subsection{Modal depth via models}
}

\improve{\cite{ShehtmanMD-md}}
\begin{proposition}\label{prop:md-lt}
If $L$ is a logic of finite modal depth, then  $L$ is locally tabular.
\end{proposition}
\begin{proof}
 Observe that for every logic, the number of pairwise non-equivalent formulas of a given finite modal depth and a given finite set of variables
 is finite (induction on the modal depth).
\end{proof}\improve{details?}
It is unknown if the converse is true. The following problem was stated in \cite{Sheht-MD-16}:
\begin{problem}\label{prob:LT-md}
Does local tabularity of $L$ imply that $L$ has the finite modal depth?
\end{problem}

\begin{proposition}
\hide{
If $L$ is locally tabular, then there is $d:\omega\to\omega$ such that
for every $k<\omega$,
$\sup\{\mdL{\vf}\mid \vf \text{ is a $k$-formula over $\Al$}\}=d(k)$.
}
If $L$ is locally tabular, then
for every finite $k$,
$\sup\{\mdL{\vf}\mid \vf \text{ is a $k$-formula over $\Al$}\}$ is finite.
\end{proposition}
\begin{proof}
 Trivial: the number of
 pairwise non-equivalent in $L$ $k$-formulas is finite.
\end{proof}
\improve{Iff? Add the syntactic staff, some axioms}

\begin{definition}\label{def:md-on-models}
In a $k$-model $M=(F,\theta)$, let   {\em $d$-equivalence} $\sim_{M,d}$ be the equivalence induced in $M$ by all $k$-formulas of depth at most $d$.
The {\em modal depth $\md{M}$ of $M$} is defined as
the least finite $d$ such that $\sim_{M,d}\,=\, \sim_{M,d+1}$, if such a $d$ exists. Otherwise, set $\md{M}=\omega$.
\end{definition}

The following is straightforward from definitions.
\begin{proposition}\label{prop:md:TFAEmod} For a $k$-model $M$ and $d<\omega$, TFAE:
\begin{enumerate}[(a)]
\item $\md{M}\leq d$.
\item $\sim_{M,d}$ is the equivalence induced in $M$ by all $k$-formulas.
\item $\sim_{M,D}\,=\,\sim_{M,D+1}$ for each $D\geq d$.
\item In $M$, every $k$-formula is equivalent to a $k$-formula of modal depth at most $d$.
\item The quotient set modulo $\sim_{M,d}$ is tuned in the frame of $M$.
\end{enumerate}
\end{proposition}

\subsection{Finite modal depth in frames}\label{sec:fmd-defs}
Our goal is to introduce the notion of modal depth for frames and their classes such that:
\begin{center}
The modal depth of the logic of a  class $\clF$  is finite iff
the modal depth of $\clF$ is finite.
\end{center}
For this, we adapt the valuation-free version of Definition \ref{def:md-on-models}.

\improve{
\IS{History: Fine 74 (k4-I), Shehtman} \IS{Check: $S5$, $DL$}}
\improve{Pseudo-syntactic criterion?}


Recall that for a set $X$  and a family $\clV$ of its subsets, $\eq{\clV}$ denotes the equivalence induced by $\clV$ on $X$:
$$a\,\eq{\clV}\, b \text{ iff }  \AA Y\in\clV \, (a\in Y \Leftrightarrow b\in Y).$$
\later{\IS{Repetition}
One can easily observe that if $\clV$ is finite, then $X{/}\eq{\clV}$ is the set atoms of the Boolean algebra generated by $\clV$ in $\pwr(X)$.
}

\improve{Done?
\IS{Let $\sim_{\clV,0}$ be $X\times X$,\IS{This is good for root, and bad for arithmetic; Fixed!}}
and $\sim_{\clV,1}$  be $\eq{\clV}$; }
Now assume that $X$ is the domain of a frame $F$ whose relations are $R_\Di$, $\Di\in \Al$.
For $d<\omega$, we recursively define the equivalence $\sim_{\clV,d}$ on $X$
and the corresponding quotient $\clV_d= X{/}{\sim_{\clV,d}}$.
Let
$\sim_{\clV,0}$  be $\eq{\clV}$;
we define $\sim_{\clV,d+1}$ as the equivalence induced by the family
$$
\clV_{d} \cup
\{
R_{\Di}^{-1} [ V ]  \mid \Di\in \Al\,
\& \, V \in \clV_{d}
\}.
$$
The families $\clV_d$ and equivalences $\sim_{\clV,d}$ are said to be {\em induced by $\clV$ in $F$}.

We also put $\clV_\omega=\bigcup_{d\in\omega} \clV_d$.\improve{ and denote the corresponding equivalence
$\sim_{\clV,\omega}$?}

The following analog of Proposition \ref{prop:md:TFAEmod} is straightforward from definitions.
\begin{proposition}
Let $\clV$ be finite.  Then we have:
\begin{enumerate}
  \item All $\clV_d$ are finite, and
  \begin{equation}\label{eq:tuned-card}
    |\clV_{d+1}|\leq |\clV_d|\cdot 2^{|\Omega|\cdot|\clV_d|}
  \end{equation}
  \item If $\clV_d=\clV_{d+1}$, then $\clV_{D+1}=\clV_D$ for all $D> d$,  and consequently, $\clV_d=\clV_{\omega}$.
  \item If $D>d$, $U\in \clV_d$  and $V\in \clV_D$, then for all $\Di\in\Omega$,
\begin{equation}\label{eq:tunedpair}
V\subseteq R_\Di^{-1}[U] \text{ of }V\cap R_\Di^{-1}[U]=\emp.
\end{equation}
  \item $\clV_d=\clV_{d+1}$ iff $\clV_d$ is tuned in $F$.
\end{enumerate}

\end{proposition}

\hide{OLD:
It is immediate that if $\clV$ is finite, then all $\clV_d$ are, and if $\clV_d=\clV_{d+1}$ for \todo{removed: some positive $d$; recheck}, then
$\clV_d=\clV_\omega$. Also,
\todo{for positive $d$, }
$\clV_d=\clV_{d+1}$ iff $\clV_d$ is tuned.
Another simple observation is that if $V\in \clV_d$ and $U\in \clV_c$ for some $c<d$, then
\begin{equation}\label{eq:tunedpair-old}
V\subseteq R^{-1}[U] \text{ of }V\cap R^{-1}[U]=\emp.
\end{equation}

}
\improve{
\IS{Give the uniform arithmetic;}}

\improve{Block-proposition?}

\begin{definition}
Let $F=(X,(R_\Di)_\Al,\clP)$ be a general frame,
and $\clV\subseteq \clP$.
For $V\in\clV_\omega$, let 
$\md{V}=\min\{d\in\omega \mid V\in \clV_d\}$.
We define the {\em modal depth $\md{\clV}$ of $\clV$ in $F$}
as
$\sup\{\md{V} \mid  V\in\clV_\omega\}$.
The {\em modal depth $\md{F}$ of $F$} is
$\sup\{\md{\clV} \mid  \clV  \text{ is a finite subset of $\clP$} \}$.
For a class $\clF$ of general frames,
$\md{\clF}=\sup\{\md{F} \mid  F\in\clF \}$.
\end{definition}

\later{
\begin{example}\label{ex:DL-S5-clusterMD}
It is trivial that the modal depth of a frame $(X,X\times X)$ is 1. (Why 1? 0!)

It is also not difficult to check that the modal depth of any point-generated $\DL$ frame $F=(X,R)$
is also 1: we have $R\cup Id_X\; =\; X\times X$, and one can check that  every partition is tuned for $R$.
\improve{details}
\end{example}
}

Hence, if $\md{\clF}$  is finite, then $\clF$ is uniformly tunable by \eqref{eq:tuned-card}, and  by Corollary \ref{cor:tunable-frame} we have:
\begin{equation}\label{eq:find-lf}
\text{If $\md{\clF}$  is finite, then $\Log(\clF)$ is locally tabular. }
\end{equation}
Our goal is to show that the modal depth of $\Log(\clF)$ is finite.


\improve{
\IS{Make a better structure of the text; say words about depth of models; make the Definition environment consistent}
}

\begin{proposition}\label{prop:md-models}
Let $M=(F,\v)$ be a $k$-model with $k<\omega$, and let $\clV=\{\v(p_i)\mid i<k\}$.
Then:
\begin{enumerate}
  \item For each $d<\omega$, $\sim_{\clV,d}\,=\,\sim_{M,d}$.
  \item If for some $d$, $\sim_{\clV,d}\,=\,\sim_{\clV,d+1}$, then every $k$-formula is equivalent in $M$ to a $k$-formula of depth at most $d$.
  \item \label{item:mds-are-equal} The modal depth of $M$ is the modal depth of $\clV$ in $F$.
\end{enumerate}
\end{proposition}
\begin{proof}
The first statement is straightforward. The second is by Proposition \ref{prop:md:TFAEmod}.
The third statement follows from the first.
\end{proof}

\begin{proposition}\label{prop:mdFrame-via-models} For a frame $F$,
$$\md{F}=\sup\{\md{M}\mid k<\omega \text{ and $M$ is a $k$-model on $F$} \}$$
\end{proposition}
\begin{proof}
 From Proposition \ref{prop:md-models}\eqref{item:mds-are-equal}.
\end{proof}

\begin{proposition}\label{prop:mdF-leq-mdL}
For a class $\clF$ of general frames,
 $\md{\clF}\leq  \md{\Log\clF}$.
\end{proposition}
\begin{proof}
Assume that $d<\md{\clF}$ and show that $d<\md{\Log\clF}$.
We have $d<\md{F}$ for some $F\in\clF$, and by Proposition \ref{prop:mdFrame-via-models} we have that  $d<\md{M}$ for some $k<\omega$ and a $k$-model $M$ on $F$.
Now $d<\md{\Log\clF}$, since $M\mo L$.
\end{proof}

The following theorem gives a semantic characterization of the modal depth of a logic.
\begin{theorem}\label{thm:md:disjclosed}
Let $\clF$ be a class of general frames closed under countable\improve{finite} disjoint sums.
Then the modal depth of $\clF$ is the modal depth of its logic.
\end{theorem}
\begin{proof}
Let $L=\Log \clF$.
By Proposition \ref{prop:mdF-leq-mdL}, $\md{\clF}\leq \md{L}$.

\smallskip

We show that $\md{L}\leq \md{\clF}$. The case $\md{\clF}=\omega$ is trivial, so assume that $\md{\clF}$ is finite.

Let $k<\omega$.
Let $I$ be the set of all $k$-formulas $\vf$ that are not in $L$. For each $\vf\in I$, let $F_\vf$ be a frame in $\clF$ invalidating $\vf$.
Consider the disjoint sum $F$ of the family $(F_\vf)_{\vf\in I}$.
Then $F\in\clF$, and for some valuation $\v$, $M=(F,v)$ is an exact $k$-model of $L$, that is for each $k$-formula $\psi$ we have: $M\mo\psi$ iff $\psi\in L$.
Note that $\md{F}\leq \md{\clF}$. By the second statement of Proposition \ref{prop:md-models}, every $k$-formula is equivalent in $M$ to a $k$-formula of modal depth at most $\md{F}$.

It follows that every formula is $L$-equivalent to a formula of modal depth of at most $\md{\clF}$.
\end{proof}

Later we show that the property of the finite modal depth of a logic is inherited from any class of Kripke frames; however, the depth can change. We also aim to prove a finite modal depth variant of the finite height criterion -- an analog of Theorem \ref{thm:clust-crit}.  These facts are more technical, and we anticipate them with the following constructions.\later{General frame issue}

\subsection{Definability of upsets in pretransitive models}

We say that $Y$ is an {\em upset} in a frame (or model) $F$, if $R[Y]\subseteq Y$ for all relations in $F$.

\smallskip

\subsubsection*{\underline{Definability lemma}}\label{subsubs:defLem}
\improve{\IS{Words? Idea}}

Fix a finite $k$ for the number of variables.

In this subsection, we assume that $M=(F,\theta)$ is a $k$-model on a pretransitive frame
$\frF=(X,(R_\Di)_{\Di\in \Al})$,
$Y$ is a non-empty upset in $M$, $\sim$ is the equivalence induced in $M$ by all $k$-formulas, and $\sim_Y$ the restriction of $\sim$ to $Y$.

\smallskip
Assume that $\ff{Y}=Y/{\sim_Y}$ is finite.

Let $M\restr Y$ be the restriction  of $M$ to $Y$, that is the $k$-model
$(\frF\restr Y, \theta \restr Y)$, where $(\theta \restr Y)(p)=\theta(p)\cap Y$.
In this case, each  $\sim_Y$-class $\tau$ is defined
in $M\restr Y$ by a $k$-formula $\alpha$; we denote it $\alpha(\tau)$.
Without loss of generality, we may assume that $\alpha(\tau)$
contains the conjunct $p_l$ or $\neg p_l$ for each $l<k$, in accordance to what is true in $\tau$.
Put
\begin{equation}\label{eq:defL:defOfPsi}
\Phi=\{\alpha(\tau)\mid \tau \in \ff{Y}\}.
\end{equation}

For a point $a$ in $Y$, let $\alpha(a)$ denote $\alpha(\tau)$, where
$a\in \tau$.
Clearly, $M\restr Y,a\mo\alpha(a)$. And since $M\restr Y$ is a generated submodel of $M$, we have
\begin{equation}\label{eq:md-basic-1}
 M,a\mo\alpha(a).
\end{equation}

\smallskip

\IS{Some $\Psi$-staff removed. Old version is stored in v.9}

Now for each $a\in Y$ we describe a formula which defines the $\sim$-class of $a$ in the model $M$.

On $\ff{Y}$, consider the minimal filtered relation $\ff{R}_{\Di}$  of $R_\Di$,
that is:
$\tau_1 \ff{R}_{\Di} \tau_2$  iff $a R_\Di b$ for some $a\in\tau_1$ and $b\in \tau_2$.

Let $\gamma(Y,\Phi)$ be the
conjunction of the following formulas:
\begin{gather}
\Box^*  \bigwedge \left\{ \alpha(\tau_1)\imp \Di \alpha(\tau_2) \mid  \tau_1,\tau_2 \in \ff{Y}, (\tau_1,\tau_2)\in \ff{R}_{\Di},\; \Di\in \Al\right\};   \label{eq:Jank1}\\
\Box^*  \bigwedge \left\{ \alpha(\tau_1)\imp \neg\Di  \alpha(\tau_2) \mid  \tau_1,\tau_2 \in \ff{Y}, (\tau_1,\tau_2)\notin \ff{R}_{\Di},\; \Di\in \Al\right\};\label{eq:Jank2}\\
\Box^* \bigvee \left\{ \alpha(\tau) \mid \tau\in \ff{Y} \right\}.\label{eq:Jank4}
\end{gather}
So $\gamma(Y,\Phi)$ is a variant of Jankov-Fine formula, where instead of variables we use formulas $\Phi$, and no root is required.

For $a\in Y$, we let $\beta(a,Y,\Phi)$ be the formula $\alpha(a)\con \gamma(Y,\Phi)$.

It is straightforward that $M\restr Y,a\mo \gamma(Y,\Phi)$, and so $M,a\mo\gamma(Y,\Phi)$. Hence,
\begin{equation}\label{eq:md-gamma-in-a}
  M,a\mo \beta(a,Y,\Phi)
\end{equation}
according to \eqref{eq:md-basic-1}.
\begin{lemma}[Definability lemma]
For each $a$ in $Y$, for each $b$ in $M$,
\begin{equation}\label{eq:defin-lemma}
M,b\mo \beta(a,Y,\Phi) \textrm{ iff } a\sim b.
\end{equation}
\end{lemma}

\begin{proof}

Within this proof, we abbreviate $\beta(a,Y,\Phi)$ as $\beta(a)$, and $\gamma(Y,\Phi)$ as $\gamma$.

The `if' direction of \eqref{eq:defin-lemma} follows from \eqref{eq:md-gamma-in-a}.

To prove the `only if', by induction on the formula structure,
for all $k$-formulas $\vf$ we show:
\begin{equation}\label{eq:same-formulas}
\text{for all $a$ in $Y$, $b$ in $M$, if } M,b\mo \beta(a), \textrm{ then } (M,a\mo \vf \tiff M,b\mo \vf).
\end{equation}
The basis of induction  follows from the definition of $\alpha(a)$. The Boolean cases are trivial. Assume that $\vf=\Di\psi$.

Let $M,b\mo\beta(a)$ and $M,a\mo \Di \psi$. We have $M,c\mo \psi$ and $a R_\Di c$ for some $c\in Y$.
By (\ref{eq:Jank1}),  we have $M,b\mo \alpha(a)\imp\Di\alpha(c)$.
Since $M,b\mo \alpha(a)$, we obtain $M,b\mo \Di \alpha(c)$. Hence, $M,d \mo\alpha(c)$ and $bR_\Di d$ for some $d$.
We have $M,b\mo\gamma$, and so
$M,d\mo\gamma$, since all conjuncts in $\gamma$ are under the scope of $\Box^*$.
 So $M,d\mo\beta(c)$. Hence $M,d\mo \psi$  by induction
hypothesis. Thus $M,b\mo \Di \psi$.

Now let $M,b\mo\beta(a)$ and $M,b\mo \Di \psi$. We have $M,d\mo \psi$ for some $d$ with $b R_\Di d$.
From (\ref{eq:Jank4}), we infer that $M,d\mo\alpha(c)$ for some $c\in Y$.
Since $M,b\mo \gamma$, we obtain $M,d\mo \gamma$. So $M,d\mo \beta(c)$. By induction
hypothesis, $M,c\mo\psi$.
We have $M,b\mo \Di  \alpha (c)$. We also have $M,b\mo \alpha(a)$.
Let $\tau_1$ and $\tau_2$ be the $\sim_Y$-classes of $a$ and $c$, respectively.
So we have $M,b\mo\alpha(\tau_1)\con \Di\alpha(\tau_2)$.
It follows from (\ref{eq:Jank2}) that $\tau_1 \ff{R}_{\Di} \tau_2$.
By the definition of $\ff{R}_\Di$, we have $a'R_\Di c'$ for some $a'\in\tau_1$ and $c'\in \tau_2$.
Since $M,c\mo\psi$, we get $M,c'\mo\psi$, so $M,a'\mo\Di\psi$, and finally $M,a\mo\Di\psi$, as required. This completes the proof of \eqref{eq:same-formulas}.
\end{proof}

\smallskip
\begin{remark}
In the case of finitely generated canonical models, this lemma can be used to define singletons in the top part of the model. In \cite{Glivenko2021}, it was used to generalize the {\em top-heavy property} (which is a classical result for the unimodal transitive case \cite[Section 8.6]{CZ}) for every logic that admits the finite height criterion.
\improve{wording; history;  \cite[Lemma~1.10]{Blok1980}}
\end{remark}

\subsubsection*{\underline{Stable top in a model}}
As before, let $\sim$ abbreviate $\sim_{M,\omega}$. 

\smallskip

\begin{proposition}\label{prop:md:gen_subfr}~
\begin{enumerate}
\item
Let $Y$ be an upset in a $k$-model $M$.
Then for each $i<\omega$ and each $a,b\in Y$, we have
\begin{equation}\label{eq:upset:1a}
  a\sim_{M,i} b \text{ iff }a\sim_{M\restr Y,i} b.
\end{equation}
If also  $\md{M\restr Y}\leq d <\omega$, then
the restriction ${\sim}\restr Y$ is the equivalence $\sim_{M\restr Y,d}$, that is:
  for each $a,b\in Y$,
\begin{equation}\label{eq:upset:c1}
a \sim b\text{ iff } a \sim_{M\restr Y,d} b.
\end{equation}
\item If $G$ is a generated subframe of $F$, then $\md{G}\leq \md{F}$.\improve{General?}
\end{enumerate}
\end{proposition}
\begin{proof}
\eqref{eq:upset:1a} is immediate from the Generated submodel lemma. \improve{\IS{Ref}}

 From this, we have $${\sim}\restr Y\;=\;(\bigcap_{i<\omega} \sim_{M,i})\restr Y\;=\;
\bigcap_{i<\omega} ({\sim_{M,i}}\restr Y)\;=\;\bigcap_{i<\omega} \sim_{M\restr Y,i}\;=\;\sim_{M\restr Y,d}.$$
Indeed, the first equality holds by the definition of $\sim_{M,\omega}$, the second follows from the definition of restriction, the third follows from \eqref{eq:upset:1a}, and the last holds because $\md{M\restr Y}\leq d$.
This
proves \eqref{eq:upset:c1}.

The second statement is straightforward from \eqref{eq:upset:1a}: if $\sim_{M,i}$ coincides with $\sim_{M,i+1}$,
then $\sim_{M\restr Y,i}$ coincides with $\sim_{M\restr Y,i+1}$.
\end{proof}

\smallskip

Now we use the Definability Lemma to state the following tool.
\begin{proposition}\label{prop:gen_subm-stable}
Assume that $M$ is a model on an $m$-transitive frame,
$Y$ an upset in $M$,
$\md{M\restr Y}\leq d <\omega$. Let
$Z\,=\,{\sim}[Y]$, that is, $Z$ is the union of all $\sim$-classes that intersect $Y$. Put $D=m+d+1$.
Then:
\begin{enumerate}
\item For each $a\in Z$, for each $b$ in $M$, \begin{equation}\label{eq:upset:1}
\text{ if  } a \sim_{M,D} b, \text{ then $a \sim b$ and hence, $b\in Z$};
\end{equation}
\item $Z$ is an upset in $M$;
\item $Z$ is definable in $M$ by a formula of modal depth $\leq D$;
\item $\md{M\restr Z}\leq D$.
\end{enumerate}
\end{proposition}
\begin{proof}
\hide{
First, we claim that the restriction ${\sim}\restr Y$ is the equivalence $\sim_{M\restr Y,d}$:
  for each $a,b\in Y$,
\begin{equation}\label{eq:upset:c1}
a \sim b\text{ iff } a \sim_{M\restr Y,d} b.
\end{equation}
To show \eqref{eq:upset:c1},
note that since $M \restr Y$ is a generated submodel of $M$, then
  for each $i<\omega$, for each $a,b\in Y$, we have
\begin{equation}\label{eq:upset:1a}
  a\sim_{M,i} b \text{ iff }a\sim_{M\restr Y,i} b
\end{equation}
 From this, we have $${\sim}\restr Y\;=\;(\bigcap_{i<\omega} \sim_{M,i})\restr Y\;=\;
\bigcap_{i<\omega} ({\sim_{M,i}}\restr Y)\;=\;\bigcap_{i<\omega} \sim_{M\restr Y,i}\;=\;\sim_{M\restr Y,d}.$$
Indeed, the first equality holds by the definition of $\sim_{M,\omega}$, the second follows from the definition of restriction, the third follows from \eqref{eq:upset:1a}, and the last holds because $\md{M\restr Y}=d$.
This
proves \eqref{eq:upset:c1}.
}

By \eqref{eq:upset:c1},
 there are only finitely many ${\sim}{\restr}Y$-classes, and we can apply Definability lemma.
Let $\Phi$  be defined according to \eqref{eq:defL:defOfPsi} in Subsection \ref{subsubs:defLem}. Since $\md{M\restr Y}\leq d$, we can assume that
$\md{\alpha}\leq d$ for each $\alpha\in \Phi$ (Proposition \ref{prop:md:TFAEmod}). It follows that $\md{\gamma(Y,\Phi)}\leq m+d+1=D$,
and so,
\begin{equation}\label{eq:upset:1d}
\md{\beta(a,Y,\Phi)}\leq D \text{ for all $a\in Y$}.
\end{equation}

To check \eqref{eq:upset:1}, assume $a\in Z$ and $a \sim_{M,D} b$.
By the definition of $Z$, we have $a\sim a'$  for some $a'\in Y$.
We have $M,a'\mo \beta(a',Y,\Phi)$.
Hence, $M,a\mo \beta(a',Y,\Phi)$.
By \eqref{eq:upset:1d},  $M,b\mo\beta(a',Y,\Phi)$. By the Definability Lemma, $b\sim a'$, and so $b\sim a$, and consequently, $b\in Z$. This proves \eqref{eq:upset:1}.

To show that $Z$ is an upset, assume that $a\sim a'\in Y$ and
$a R_\Di b$ for some $\Di\in\Al$ and some $b$. We have $M,a \mo \gamma(Y,\Phi)$, which has two consequences:
firstly, $M,b \mo \gamma(Y,\Phi)$ (by the definition of $\gamma$ and pretransitivity of the frame of $M$),
and secondly, $M,b\mo \alpha(b')$ for some $b' \in Y$ (due to \eqref{eq:Jank4} in the  definition of $\gamma$). So,
$M,b\mo \beta(b',Y,\Phi)$, and by the Definability Lemma, $b'\sim b$. So $b\in Z$, as required.

It is immediate that $Z$ is definable by a finite disjunction of formulas $\beta(a,Y,\Phi)$: simply make $a$ range over representatives of ${\sim}\restr Y$-classes.

According to \eqref{eq:upset:1}, we have $a \sim_{M,D} b$ iff $a \sim_{M,D+1} b$
for all $a,b \in Z$.
Since $Z$ is an upset in $M$, by \eqref{eq:upset:1a} we have $\md{M\restr Z}\leq D$.
\end{proof}

 \improve{
\IS{
  If $\md{\clF}$ is finite, then $\clF$ is pretransitive.
}
}

\subsection{Computing the modal depth: basic semantic tool}

For a class $\clF$ of frames, its {\em transitivity index $\tra{\clF}$} is defined as the
transitivity index $\tra{L}$ of its logic $L=\Log{\clF}$.
\begin{theorem}\label{thm:md-basic}
Let $\clF$ be a class of 
Kripke
frames, $L$ its logic.
Then $$\md{L}\leq \md{\clF}+\tra{\clF}+1.$$
\end{theorem}
\begin{proof}
Let  $\md{\clF}=d$. Assume that $d$ is finite, for otherwise there is nothing to prove.
In this case the logic $L$ of $\clF$ is locally tabular by \eqref{eq:find-lf},
and so  $m=\tra{\clF}$ is finite as well (Proposition \ref{prop:pretr}).

Let $k<\omega$. Like in Theorem \ref{thm:md:disjclosed}, we consider an exact $k$-model
$M$ of $L$, which is the disjoint sum of a family of $k$-models $(M_i)_{i\in I}$, whose frames $F_i$ are in $\clF$.
Let $Y_i$ be the domain of $M_i$.

Each $Y_i$ is an upset in $M$. We also have $\md{M_i}\leq d$ for all $i\in I$ (Proposition \ref{prop:mdFrame-via-models}).
Let $Z_i\,=\,{\sim}[Y_i]$, where as usual $\sim$ abbreviates $\sim_{M,\omega}$.
Put $D=m+d+1$.
By Proposition \ref{prop:gen_subm-stable}, for each $i\in I$, for each $a$ in $Z_i$  and $b$ in $M$, we have:
\begin{equation}\label{eq:md:themDisj}
\text{ If  } a \sim_{M,D} b, \text{ then } a \sim b.
\end{equation}
Note that the union $\bigcup_I Z_i$ contains all points in $M$. Hence, \eqref{eq:md:themDisj} holds for all $a,b$ in $M$.
Thus, $\sim_{M,D}\;=\;\sim$, and so $\md{M}\leq D$.
By Proposition \ref{prop:md:TFAEmod}, every $k$-formula in $M$ is equivalent to a formula of modal depth $\leq D$. Since $M$ is an exact $k$-model of $L$, the theorem follows.
\end{proof}

From this theorem and Proposition \ref{prop:mdF-leq-mdL}, we obtain:
\begin{corollary}
The modal depth of the logic of a  class $\clF$  of
Kripke
frames is finite iff
the modal depth of $\clF$ is finite.
\end{corollary}

\later{\IS{ARITHMETIC!! DC}
\begin{example}
  Due to this theorem and Example \ref{ex:DL-S5-clusterMD}, we obtain that $\md{\LS{5}}\leq 2$ and $\md{\DL}\leq 2$. The first bound is not optimal, since $\md{\LS{5}}=1$, while $\md{\DL}=2$ due to \cite{Sheht-MD-16}.
  \improve{wording}
\end{example}
}

\smallskip

\subsection{Computing the modal depth via height and clusters}~

\smallskip

Our next aim is to prove an analog of the cluster criterion, Theorem \ref{thm:clust-crit}, for the modal depth.

\smallskip

Let $\min F$ denote the union of minimal clusters in $F$.

\begin{lemma}\label{lem:uni-top-down}
Let $Y$ be an upset in an $m$-transitive Kripke frame $F$, and let the  complement of $Y$ be  included in $\min F$. Assume that
$\md{(F\restr \min F)}\leq c$,
$\md{F\restr Y}\leq d$.
Then \begin{equation}\label{eq:md-up-down-lemma}
\md{F}\leq d+m+c+1.
\end{equation}
\end{lemma}
\begin{proof}
We consider the only interesting case when $c,d,m$ are finite. Put $E=d+m+c+1$.
Consider
$k<\omega$ and a $k$-model $M=(F,\v)$. By Proposition \ref{prop:mdFrame-via-models}, it suffices to show that
\begin{equation}\label{eq:topdown-1}
 \md{M}\leq E.
\end{equation}
It turn, for \eqref{eq:topdown-1}, it is enough to show that the set of $\sim_{M,E}$-classes is a tuned partition of $F$ (Proposition \ref{prop:md:TFAEmod}). This is our goal.

\medskip

Firstly, we  analyze an upper part of the model $M$.

We start with a simple observation about the equivalences on $Y$.  As before, let $\sim$ denote $\sim_{M,\omega}$.
By the first statement of Proposition \ref{prop:md:gen_subfr}, we have:
\begin{eqnarray}
   &&\text{For each $i\geq d$, we have ${\sim_{M,i}}\restr Y={\sim_{M,d}}\restr Y$, and} \label{eq:topdown-Y-i}
   \\
   \label{eq:topdown-Y}
   &&{\sim}\restr Y\;=\;{\sim_{M,d}}\restr Y  \;=\;{\sim_{M,E}}\restr Y.
\end{eqnarray}
Hence, the sequence of equivalences $\sim_{M,i}$ stabilizes on $Y$ at $i=d$.

Next, we 
stabilize equivalences on the set \mbox{$Z\;=\;{\sim}[Y]$}.
By Proposition \ref{prop:gen_subm-stable}, we have: $Z$ is an upset in $M$;  on $Z$,
$\sim$ coincides with $\sim_{M,D}$ for $D=d+m+1.$
Clearly, $E\geq D$, so
$${\sim}\restr Z \;=\; {\sim_{M,D}}\restr Z\;=\;{\sim_{M,E}}\restr Z.$$
Hence, equivalences $\sim_{M,i}$ stabilize on $Z$ at $i=D$.

Let $\clS$ be the family  of ${\sim}$-classes that are included in $Z$.
By Proposition \ref{prop:md:TFAEmod},
\begin{equation}\label{eq:S-tuned}
\text{$\clS$ is a tuned partition of $F\restr Z$.}
\end{equation}
This means that we have achieved our goal for the upper part $M\restr Z$ of $M$.

\medskip
Secondly, we consider the remaining, bottom, part of $M$.

Let $X$ be the complement of $Z$ in $F$.
Since $Z$ in an upset in $F$, $X$ is a downset in $F$. Also, it is contained in $\min F$.
In follows that $X$ is a union of some minimal clusters in $F$. Hence, $G=F\restr X$ is a generated subframe of
$F\restr \min F$. By the second statement of Proposition \ref{prop:md:gen_subfr},
$$\md{G}\leq c.$$

\medskip
Now we will combine above observations to get a tuned partition of the whole model $M$.
Note that for each  $i\geq D$, we have:
\begin{equation}\label{eq:topdown-2}
\text{Every $\sim_{M,i}$-class is a subset of either $Z$  or $X$.}
\end{equation}
This is another consequence of Proposition \ref{prop:gen_subm-stable}: $Z$ in definable in $M$ by a formula of modal depth $\leq D$.

Put $\clV=\{\v(p_i)\mid i<k\}$.
Let $\clV_i$ and $\sim_{\clV,i}$ denote the partitions and equivalences induced by $\clV$ in $F$, as defined in Section \ref{sec:fmd-defs}.
By Proposition \ref{prop:md-models},  $\sim_{\clV,i}$ is $\sim_{M,i}$, and so
$\clV_i$ is the quotient set of the domain of $M$ modulo $\sim_{M,i}$.

We claim that for each $a,b$ in $X$, $\Di\in \Al$, we have:
\begin{equation}\label{eq:topdown-2b}
 \text{If $a\sim_{M,d+1} b$ and $U\in \clS$, then  ($a\in R_\Di^{-1}[U]$ iff $b\in R_\Di^{-1}[U]$).}
\end{equation}
\improve{Rewrite in form of \eqref{eq:tunedpair}!}
Indeed, if $a R a'$ for some  $a\in X$ and $a'\notin X$, then $a'$ is not in the same cluster with $a$ (recall that $X$ is a union of clusters). In this case $a'$ does not belong to $\min F$. Hence, $a'\in Y$.
Now \eqref{eq:topdown-2b}  follows from \eqref{eq:topdown-Y-i} and the property \eqref{eq:tunedpair}.

According to \eqref{eq:topdown-2}, for $i\geq D$, we have:
\begin{equation}\label{eq:clV-containsS}
\text{$\clV_i$ is the disjoint union of
$\clS$ and the family
$\clU_i=\{V\in \clV_i\mid V\subseteq X\}.$}
\end{equation}

Let $\clW=\clU_{D}$.
Consider the partitions $\clW_i$ and equivalences $\sim_{\clW,i}$  induced by $\clW$ in $G$.
By induction on $i$, we show that
\begin{equation}\label{eq:topdown-3}
\clW_i=\clU_{D+i}.
\end{equation}
The basis holds, since $\clW$ is a partition of $X$, and so $\clW_0=\clW=\clU_D$.
Assume $i>0$. From the induction hypothesis and \eqref{eq:topdown-2b}, it follows that
for each $a,b\in X$,
$$
a \sim_{\clW,i} b  \text{ iff }  a \sim_{\clV,D+i} b.
$$
So the quotients of $X$ modulo $\sim_{\clW,i}$ and modulo $\sim_{\clV,D+i}$ coincide.
The former quotient is $\clW_i$, and the latter is $\clU_{D+i}$, which proves the inductive step.

Since $\md{G}\leq c$, $\clW_c$ is tuned in $G$, and so by \eqref{eq:topdown-3},
\begin{equation}\label{eq:clU-tuned-inG}
\text{$\clU_{D+c}$  is tuned in $G$.}
\end{equation}
By  \eqref{eq:clV-containsS},
$\clV_{D+c}$ is the disjoint union of
$\clS$ and
$\clU_{D+c}$.
From \eqref{eq:clU-tuned-inG}, \eqref{eq:S-tuned}, and the unitive \eqref{eq:topdown-2b}, it easily follows that $\clV_{D+c}$ is tuned in $F$.
\end{proof}
\todo{
Looks like this proof gives $$\md{F}\leq d+1+\max\{m,c\}:$$
~\\
After $d+1+c$ -- all from $X$ is tuned in $F$; \\
After $d+1+m$ -- all from $Z$ is tuned in $F$.

}

\begin{theorem}\label{thm:md:clusters}
Let $\clF$ be a class of Kripke frames such that $\h{\clF}\leq h$ for a positive finite $h$.
Assume that the modal depth $d$ of the logic $\Log\clust{\clF}$ is finite. Then the modal depth of $\Log\clF$ is also finite, and specifically
\begin{equation}\label{eq:MD-clusters}
\md{\Log\clF}\leq (d+m+1)h-m-1,
\end{equation}
where $m=\tra{\clF}$.
\end{theorem}
\begin{remark}
It is not given explicitly that $m$ is finite. 
However,
since the modal depth of the logic $\Log\clust{\clF}$ is finite, this logic
is locally tabular by Proposition \ref{prop:md-lt}, and so $\clust{\clF}$
is pretransitive by Proposition \ref{prop:pretr}.
It is easy to see that pretransitivity of clusters combined with the finite height of $\clF$ implies pretransitivity of $\clF$. So $m$ is finite.
\end{remark}
\begin{proof}
Let $\clF_n$ be the class of all $m$-transitive frames of height at most $n$ and whose clusters are in $\clust{\clF}$. Put $f(n)=(d+m+1)n-m-1$.
By induction on $n$, we show that for $n>0$,
\begin{equation}\label{eq:ind-in-fmd-cl}
\md{\Log{\clF_n}}\leq f(n).
\end{equation}

Let $n=1$. Observe that a cluster in an $m$-transitive frame is $m$-transitive. Hence,
$\clF_1$ is the class of all disjoint sums of clusters from $\clust{\clF}$, and so $\Log \clF_1 =\Log{\clust{\clF}}$. Hence, $\md{\Log{\clF_1}}=d=f(1)$, as required.

Assume $n>1$. Consider $F\in \clF_n$. Let $Y$ be the complement of $\min F$ in $F$.
Then $F\restr Y\in \clF_{n-1}$, so by induction hypothesis, $\md{F\restr Y}\leq f(n-1)$.
The frame $F\restr \min F$ is a disjoint sum of clusters from $\clust{\clF}$, and so $\md{(F\restr \min F)}\leq d$.
 By Lemma \ref{lem:uni-top-down}, $\md{F}\leq f(n-1)+m+d+1=f(n)$. It follows that $\md{\clF_n}\leq f(n)$.
Since $\clF_n$ is closed under disjoint sums, we have $\md{\clF_n}=
\md{\Log{\clF_n}}$ by Theorem \ref{thm:md:disjclosed}. This proves \eqref{eq:ind-in-fmd-cl}.

Since $\clF\subseteq \clF_h$, we have $\Log\clF_h\subseteq \Log\clF$, and so
$\md{\Log\clF}\leq \md{\Log\clF_h}$.
\end{proof}

\improve{
\begin{corollary}
Let $\clF$ be a class of $m$-transitive Kripke frames of uniform finite height $h>0$,  and assume that $\md{\clust{\clF}}\leq d$.
Then
\begin{equation}\label{eq:MD-clusters}
\md{\Log\clF}\leq  ....
\end{equation}
where $m=\tra{\clF}$.
\end{corollary}
}

This theorem is a helpful tool for finite modal depth results.
For a pretransitive logic $L$, let $L[h]$ be the extension of $L$ with the finite height formula $B_h^*$.

\begin{example}
To the best of our knowledge, the sharpest known upper bound for logics above
$\LS{4}$ is $\md{\LS{4}[h]}\leq 4h-3$ \cite{Sheht-MD-16}.
This bound can be refined due to Theorem \ref{thm:md:clusters}. Namely,
the transitivity index $m$ of these logics is 1.
The clusters in preorders are frames of the form $(X,X\times X)$. Their logic is $\LS{5}$, whose modal depth is 1. Hence,
$\md{\LS{4}[h]}\leq 3h-2.$

For the Grzegorczyk logic $\Grz$, the estimate \eqref{eq:MD-clusters} in Theorem \ref{thm:md:clusters} coincides with that of
\cite{Sheht-MD-16}:
$\md \Grz[h]\leq 2h-2$.  Due to a more recent result \cite{DiplomMaroch}, $\md \Grz[h]=2h-2$, and so this estimate is optimal in this certain case. However, we do not claim that  \eqref{eq:MD-clusters}
is optimal in general.\footnote{
From the proof of Lemma \ref{lem:uni-top-down}, one can obtain the bound $\md{F}\leq d+1+\max\{m,c\}$,
which would be used for a refinement of \eqref{eq:MD-clusters}.
} 
\end{example}

\begin{example}
We claim that all 1-transitive unimodal locally tabular logics have the finite modal depth.

Consider the logic $\wK4{}$ given by the formula $\Di\Di p \imp \Di p\vee p$.
This is the least 1-transitive logic.
It is known to be the logic of the class of frames $(X,R)$ where $R\cup Id_X$ is transitive.

The logic of clusters in the case of $\wK4{}$-frames is the difference logic $\DL{}$. Its modal depth is 2 \cite{Sheht-MD-16}. Hence: $$\md{\wK4[h]}\leq 4h-2.$$

By Theorem \ref{thm:LTforLm}, $\wK4{}$  admits the finite height criterion. So above $\wK4{}$, local tabularity, finite modal depth, and finite height are equivalent.
\end{example}
\hide{
\begin{remark}
To the best of our knowledge, this example is knew.
\end{remark}
}

\improve{$\md{\LK{}+\Di p\imp \Box\Di p}\leq 5$.}

\hide{
\begin{corollary} For $h>0$, we have:
\begin{enumerate}
\item $\md{\LS{4}[h]}\leq 3h-2$;
\item $\md{\wK4[h]}\leq 4h-2$;
\end{enumerate}
\end{corollary}
\begin{proof}
The transitivity index $m$ of these logics is 1.

The clusters in preorders are frames of form $(X,X\times X)$. Their logic is $\LS{5}$, whose modal depth is 1. Now the first result follows.

The logic of clusters in the case of $\wK4{}$-frames is the difference logic $\DL{}$. Its modal depth is 2 \cite{ShehdmanMD-DL}. This implies the second result.
\end{proof}
}

Theorem \ref{thm:md:clusters} provides the following version of the finite height criterion for modal depth.
\begin{corollary}\label{cor:md-cluster-crit}
Let $\clF$ be a class of Kripke frames such that the class $\clust{\clF}$ has finite modal depth.
Then
$\clF$ has finite modal depth iff it has finite height.
\end{corollary}

This leads to the following equivalent reformulation of Problem \ref{prob:LT-md}.
\begin{problem}
Let $\clC$ be a uniformly tunable class of cluster-frames.  Does $\clC$ have the finite modal depth?
\end{problem}
This  problem is equivalent to Problem \ref{prob:LT-md} according to Corollary  \ref{cor:md-cluster-crit} and the characterization of locally tabular logics given in Corollary \ref{cor:tuned-for-frames}.
\later{
Note that the affirmative solution for these problems would imply that all logics admit the $k$-finiteness condition -- thus providing a solution of Problem \ref{prob:k-fin}.
 --- WHAT IS  $k$?
}

\improve{
Prove that every logic with finite md is the logic of a class of frames of uniform f.h. whose clusters have fmd (canonical model). It might be useful to state a separate proposition for the bottom part of a model $M$: If top is $D$-definable, $D>d$, and md of top $\leq d$, and the bottom has md $\leq c$, then md of $M$ is $\leq c+d+m+1$.
}

\section{Acknowledgements}

I am grateful to Valentin Shehtman for valuable discussions.

The work on this paper was supported by NSF Grant DMS - 2231414.

\bibliographystyle{amsalpha}
\bibliography{LT-FinH}
\end{document}